\documentclass[10pt]{amsart}
\usepackage{amsmath}
  \usepackage{graphics} 
\usepackage{graphicx}  
  \usepackage{epsfig} 
 \usepackage[colorlinks=true]{hyperref}
  \usepackage{float}
  \usepackage{caption}
  \usepackage[english]{babel}
\usepackage{amsmath}
\usepackage{wrapfig}
\usepackage{graphicx}
\usepackage{subcaption}
\usepackage{stackengine}
\newcommand\dhookrightarrow{\mathrel{%
  \ensurestackMath{\stackanchor[.1ex]{\hookrightarrow}{\hookrightarrow}}
}}
\usepackage{lscape}
\usepackage{amssymb}

\usepackage{amsthm}
\theoremstyle{plain}

\newtheorem*{theorem*}{Theorem}
\newtheorem{theorem}{Theorem}[section]

\newtheorem{lemma}[theorem]{Lemma}

\newtheorem{conjecture}{Conjecture}

\theoremstyle{definition}
\newtheorem{definition}[theorem]{Definition}
\newtheorem{remark}{Remark}

\title[ FTE competition ] 
    {The effect of ``very fast" strategies on two species competition}


\begin{document}
\maketitle

\centerline{\scshape  Aniket Banerjee$^{1}$, Vaibhava Srivastava$^{1}$ and Rana D. Parshad$^{1}$}
\medskip
{\footnotesize

   \medskip
   
    \centerline{1) Department of Mathematics,}
 \centerline{Iowa State University,}
   \centerline{Ames, IA 50011, USA.}

 }

\begin{abstract}


 We consider the effect of finite time extinction mechanisms (FTEM) such as (1) semi-linear harvesting terms, (2) quasi-linear fast diffusion terms
 on two species Lokta-Volterra competition models. We show that these mechanisms can alter classical dynamics of competitive exclusion, and weak and strong competition by acting only on a \emph{small} portion of the weaker competitors' population,
analogous to small defector populations in game theory \cite{DC23}.
 In particular, a stronger competitors population, with a few individuals dispersing (``defecting") very quick, could exhibit bi-stability, as well as competitive exclusion \emph{reversal}. The non-linear harvesting is applied to aphid-soybean crop systems, wherein novel dynamics are observed. Applications  to bio-control of invasive pests such as the soybean aphid are discussed. 

\end{abstract}

 \section{Introduction}

The classical two-species Lotka-Volterra competition model has been rigorously investigated in the literature \cite{LY23, Cantrell2003}. It models two competing species via considering growth  inter and intra species competition \cite{Brown1980, Lou2008}. It predicts well-observed states in population biology, of co-existence, competitive exclusion of one competitor, and bi-stability - and finds numerous applications in population ecology, invasion science, evolutionary biology, epidemics, economics and game theory, \cite{Lou2008, Cantrell2003, Cantrell2004, CC18, DC23, LY23}. Note, the equilibrium states are achieved only asymptotically, as is the case in many differential equation population models.
Competition theory predicts that two competing species whose niches overlap perfectly, cannot coexist, and one must competitively exclude the other. The classical Lokta-Volterra ODE competition model predicts this dynamic under certain parametric regimes. Other parametric regimes predict initial condition-dependent outcomes (strong competition) or stable coexistence (weak competition). Analogously, the spatially explicit Lokta-Volterra model predicts that the slower diffusing competitor will always exclude the faster one, given that they have the same kinetics.

Invasive species are most successful in environments where they lack close relatives \cite{brandt09, D59}. This has been confirmed in experiments, where invasion success in microbial communities \emph{increases} as phylogeny or ``relatedness" between invader and invadee decreases \cite{jiang2010experimental}. Thus an invasion is most likely to be successful, if the invasive species faces low intraspecies competition, whilst being a superior interspecies competitor \cite{violle2011phylogenetic, O89}. 
 However in reality a successful invasion results from the interaction of a myriad of factors, biological, environmental, landscape, and temporal - this is quantified via the concept of a species \emph{ecological niche}\cite{chesson2000mechanisms, holt2009bringing}.
To fix ideas, \emph{niche space} can be thought of as a continuum, similar to $\mathbb{R}^{4}_{+}$, where the four axes are physical and environmental factors (say climate), biological factors (say resources), space and time. A species \emph{niche} is the \emph{response} it has to each point in the space, and the \emph{effect} it has at each point \cite{chesson2000mechanisms, shea2002community}. The niche-based hypothesis posits invasive species are dexterous at using unexploited resources - that is filling \emph{vacant} niches, or broadening their niche breadth if the opportunity presents itself \cite{shea2002community, lekevivcius2009vacant}. 

In \cite{parshad2021some} we decided to consider the effect of creating a vacant niche very rapidly, via finite time extinction mechanisms (FTEM). There are several motivations for studying FTEM. In classical biological control, pest populations can \emph{rebound} from levels as close to extinction as one pleases \cite{P19}. This is well observed with soybean aphids (\emph{Aphis glycines}), an invasive pest on soybean crops, particularly in the North-central US.
Recent work in epidemics has considered a class of susceptible-infected models with non-smooth incidence functions, that can lead to host extinction in finite time - yet are seen to be better fitted to data collected than smooth systems \cite{FCGT18, FT18, G08, F02, D97}. Non-smooth responses have been considered analytically in the predator-prey literature as well \cite{B12, S97, S99, BS19, KRM20}, and in fitting  various data \cite{L12, R05, I08, M04}.

Consider a model, for the competition dynamics between a strong(er) invader $u$ and weak(er) invadee $v$, permitting multiple locally stable equilibria. The state $(u^{*}, 0)$ would imply invasion success \cite{L16}, or the competitive exclusion of $v$. Managers would aim to rather attain $(0,v^{*})$ - eradication or $(u^{*}, v^{*})$ - where $u^{*}$ was at a \emph{manageable} level. These states would imply invasion failure \cite{chesson2000mechanisms, L16}. We show one can \emph{reverse} invasion success by strategically \emph{increasing} the rate of attraction to the weak(er) invadees extinction state, thereby creating a niche opportunity for the invader, occupied prior by the invadee. This will indirectly \emph{increase} intraspecific competition among the invaders as they attempt to fill the niche and facilitate invasion failure.

In the current manuscript, we show the following,

\begin{itemize}

\item A two species ODE Lotka-Volterra competition model, with a semi-linear harvesting term exhibiting FTEM can lead to bi-stability, via Lemma \ref{stable_equi} and Theorem \ref{th:stability}, also see Fig. \ref{int_stable}. FTEM can also enable a competitor to avoid competitive exclusion and persist, for various regimes of initial conditions. This is seen via Theorem \ref{thm:l2}, also see Fig. \ref{no_int}. To this end, the harvesting needs to be performed only on a portion of the weaker competitors population. This is seen via Theorem \ref{th:existence}.

\item A new spatially explicit model is introduced, in the form of a quasi-linear PDE. Herein a certain portion of the weaker population moves at a faster rate than the stronger population. Degeneracy theory is applied to understand solutions to this model system, see definition \ref{def:weak}. Also see Theorem \ref{classic_pde} and Fig. [\ref{fig:pde_ce_sim1}(A),\ref{fig:pde_ce_sim2}(A)].

\item FTEM in this PDE model, can cause the slower diffuser to \emph{loose}, via Theorem \ref{FFTE}, as long as there are a few ``very" fast-moving individuals. Refer to Fig. [\ref{fig:pde_ce_sim1}(B),\ref{fig:pde_ce_sim2}(B)].

\item This framework is applied to a recent soybean aphid-soybean plant model, where novel dynamics are observed. See  Fig. \ref{aphid_vir}.


\end{itemize}
\section{Prior Results}

\subsection{General competition Model}

Consider the classical two-species Lotka-Volterra ODE competition model,

\begin{equation}
\label{Eqn:1}
\frac{du}{dt} = a_1 u - b_1 u^2 - c_1 uv, \    \frac{dv}{dt} = a_2 v - b_2 v^2 - c_2 uv.
\end{equation}

where $u$ and $v$ are the population densities of two competing species, $a_1$ and $a_2$ are the intrinsic (per capita) growth rates, $b_1$ and $b_2$ are the intraspecific competition rates, $c_1$ and $c_2$ are the interspecific competition rates. All parameters considered are positive. The dynamics of this system are well studied \cite{Murray93}. We recap these briefly,

\begin{itemize}
    \item $E_0 = (0,0)$ is always unstable.
    \item $E_u = (\frac{a_1}{b_1},0)$ is globally asymptotically stable if $\dfrac{a_{1}}{a_{2}} > \max\left\lbrace\dfrac{b_{1}}{c_{2}},\dfrac{c_{1}}{b_{2}}\right\rbrace$. Herein $u$ is said to competitively exclude $v$.
    \item $E_v = (0,\frac{a_2}{b_2})$ is globally asymptotically stable if $\dfrac{a_{1}}{a_{2}}<\min\left\lbrace\dfrac{b_{1}}{c_{2}},\dfrac{c_{1}}{b_{2}}\right\rbrace$. Herein $v$ is said to competitively exclude $u$. 
    \item $E^* = \Big(\frac{a_1b_2-a_2c_1}{b_1b_2-c_1c_2},\frac{a_2b_1-a_1c_2}{b_1b_2-c_1c_2}\Big)$ exists when $b_1b_2-c_1c_2 \neq 0$. The positivity of the equilibrium holds if $\frac{c_2}{b_1}<\frac{a_2}{a_1}<\frac{b_2}{c_1}$ and is globally asymptotically stable if $b_1b_2-c_1c_2>0$. This is said to be the case of weak competition.
   
   \item If $b_1b_2-c_1c_2<0$, then $E^* = \Big(\frac{a_1b_2-a_2c_1}{b_1b_2-c_1c_2},\frac{a_2b_1-a_1c_2}{b_1b_2-c_1c_2}\Big)$ is unstable as a saddle. In this setting, one has an initial condition dependent attraction to either $E_u(\frac{a_1}{b_1},0)$ or $E_v(0,\frac{a_2}{b_2})$. This is the case of strong competition.
\end{itemize}

\subsection{Prior work on FTEM}
The following model was introduced in \cite{parshad2021some} as a means to show that finite time extinction mechanism (FTEM) can alter the above classical dynamics.

\begin{equation}
\label{eq:Ge1}
\left\{ \begin{array}{ll}
\dfrac{du }{dt} &~ = a_{1}u - b_{1}u^{2} - c_{1}u^{p} v ,   \ 0 < p  \leq 1,         \\[2ex]
\dfrac{dv }{dt} &~ =  a_{2} v - b_{2} v^{2} - c_{2} uv^{q}, \ 0 < q \leq 1.
\end{array}\right.
\end{equation}

We see that the classical model is a special case of the above when $p=q=1$.
Note, $0 < p < 1, q=1$, allows for \emph{finite} time extinction of $u$,
 and $ p=1, 0 < q < 1$, allows for \emph{finite} time extinction of $v$. 
  There is also the more complex case when $0 < p < 1, 0 < q < 1$. If $p=q=1$, and $\dfrac{a_{1}}{a_{2}} > \max\left\lbrace\dfrac{b_{1}}{c_{2}},\dfrac{c_{1}}{b_{2}}\right\rbrace$, then $u$ is said to competitively exclude $v$, and $(\frac{a_1}{b_1},0)$ is globally asymptotically stable. One can investigate the effect of $0<q<1$ on this situation. To this end
 the following result \cite{parshad2021some} is recapped,
 \begin{theorem}
\label{thm:l2}
Consider \eqref{eq:Ge1}, with $p=1$, $\frac{a_{1}}{c_{1}} > \frac{a_{2}}{b_{2}}$, $(a_{2})^{2}b_{1} + 2a_{2}c_{1}c_{2}> 4a_{1}b_{2}c_{2}$. Then there exists a $q^{*} \in (0, 1)$, s.t. for any $q^{*} < q <1$ there is no interior equilibrium, for $q=q^{*}$ there is a unique non-hyperbolic equilibrium and for $0 < q < q^{*} $, there exists two interior equilibrium, a saddle, and a nodal sink.  
\end{theorem}

\begin{remark}
It is a natural question to investigate the validity of the $\boxed{-c_{2} uv^{q}, \ 0 < q \leq 1}$ term, as a control mechanism in a laboratory setting.
Since $-c_{2} uv$ models classical inter-species competition, it is inherent to any system, and not pragmatic to manipulate. This raises the question of other reaction terms, that would be more amenable to laboratory manipulation.
\end{remark}

\subsection{The modified Model}

Consider the population dynamics of a species $v$ typically governed by an ODE,
$\frac{dv}{dt} =  f(v,u)$, where  dynamics such as growth, death, competition, and depredation are embedded in $f(v,u)$. One can define an operator $ \mathcal{L}: \mathbb{R} \mapsto \mathbb{R}$, which is a combination of a linear growth operator such as
 $ \mathcal{L}(v) = a_{1} v$, and a non-linear harvesting type operator such as where $\boxed{L^{*}(v) \approx -v^{p}}$, $0<p<1$.
We take the following approach. Define, 
\begin{eqnarray}
&&\mathcal{L}^{1}(v)  \nonumber \\
&=& (q\mathcal{L} + (1-q)L^{*})(v), \ 0<q<1  \  \nonumber \\
&=& q\mathcal{L}( v) + (1-q)L^{*}(v), \ 0<q<1 \  \nonumber \\
&=& (q)a_{1}v -  (1-q) v^{p-1}v),  \nonumber \\
\end{eqnarray}
 where $\boxed{L^{*}(v) \approx -v^{p}}$, $0<p<1$. Thus the operator $L^{*}(v)$ will play the role of the sub-linear harvesting term, so as to cause finite time extinction. However, this is only effective on a fraction $\left((1-q), q \in (0,1)\right)$ of the population.

\begin{remark}
Thus the growth operator $\mathcal{L}^{1}$, provides a way to formalize the action via which we have a fraction of the population growing as per the regular growth coefficient, while the other fraction is harvested at a sub-linear rate.
\end{remark}

\begin{remark}
Another interpretation of the non-linear operator $L^{*}$ could be a faction or group within the $v$ population that decides to divert from the incentives of the group as a collective whole. Such competitive mechanisms have been under intense recent investigation.
\end{remark}

The model is modified to bring in the effect of finite time extinction. The modified model is as follows:

\begin{equation}
\label{Eqn:2}
\begin{aligned}[t]
\frac{du}{dt} &= a_1 u - b_1 u^2 - c_1 uv, \\
\frac{dv}{dt} &= a_2 q v - b_2 v^2 - c_2 uv - (1-q) v^p.
\end{aligned}
\end{equation}

where $0 < p < 1$ and $0 < q <1 $.\\

\section{Equilibria and Stability of System (\ref{Eqn:2})}

\subsection{Existence of equilibria}\hfill\\

In this subsection, we perform the qualitative analysis of the system (\ref{Eqn:2}). Considering the biological implication of the system on ecological population, we are interested to study the dynamics of the system (\ref{Eqn:2}) in the closed first quadrant $\mathbb{R}^2_+$ in the (u,v) plane. It is obvious that $E_0(0,0)$ is an equilibrium point of the system (\ref{Eqn:2}). We can have two type of boundary equilibria $E_u(\frac{a_1}{b_1},0)$ which is trivial and $E_v(0,\Bar{v})$ as described in lemma \ref{Boundary_v}. The whole plane $\mathbb{R}_+^2$ is not the positively invariant under parameter $q$ for the system (\ref{Eqn:2}). So, the positively invariant subspace we have from classical model $\Gamma \{(u,v) : 0 \leq u \leq \frac{a_1}{b_1}, 0 \leq v \leq \frac{a_2}{b_2}\}$ contains all the equilibria of the system (\ref{Eqn:2}).

In this section, we will discuss the existence of the interior equilibria of the system (\ref{Eqn:2}) in the invariant set. In order to get the equilibria of the system (\ref{Eqn:2}), we consider the nullclines of both the population $u$ and $v$ of the system (\ref{Eqn:2}), which are given by:\\
\begin{equation}
\label{nullclines}
    \begin{aligned}
    a_1 u - b_1 u^2 - c_1 uv &= 0,\\
    a_2 q v - b_2 v^2 - c_2 uv - (1-q) v^p &= 0.
    \end{aligned}
\end{equation}

Simplifying system of equations (\ref{nullclines}) to get an explicit form in terms of population $v$ we get,
\begin{equation}
\label{nullcline}
    (a_2b_1q-c_2a_1)+(c_1c_2-b_1b_2)v-b_1(1-q)v^{p-1}=0
\end{equation}
We will study the dynamics of the polynomial (\ref{nullcline}) to find the number of equilibria possible for the system (\ref{Eqn:2}).

\begin{lemma}
 \label{Boundary_v}
 Let $v_{\phi}^{p-2}=\frac{b_2}{(1-q)(1-p)}$. Then we have:
 \begin{enumerate}
     \item If $q < \frac{b_2v_\phi(2-p)}{a_2(1-p)}$ then there does not exist any boundary equilibrium of the form $E_v(0,\Bar{v})$.
     \item If $q = \frac{b_2v_\phi(2-p)}{a_2(1-p)}$ then there exists an unique boundary equilibrium $E_v(0,\Bar{v})$.
     \item If $q > \frac{b_2v_\phi(2-p)}{a_2(1-p)}$ then there exists two boundary equilibrium $E_{v_1}(0,\Bar{v_1})$ and $E_{v_2}(0,\Bar{v_2})$.
 \end{enumerate}
\end{lemma}

\begin{lemma}
\label{b_1b_2-c_1c_2<0 exists}
 If $b_1b_2-c_1c_2 \leq 0$ there exists only one unique interior coexistence equilibrium.
\end{lemma}

 \begin{figure*}
    \begin{subfigure}[b]{.475\linewidth}
        \includegraphics[width=6cm, height=6cm]{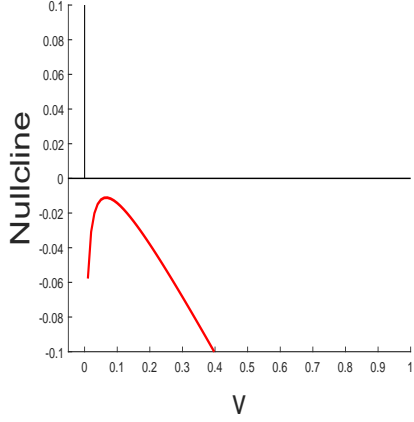}
        \caption{The parameters used are  $a_1=0.4,a_2=0.4,b_1=1,b_2=0.6,c_1=0.3,c_2=0.8,p=0.6,q=0.98$. }
        \label{no_int}
    \end{subfigure}
    \hfill
    \begin{subfigure}[b]{.475\linewidth}
        \includegraphics[width=6cm, height=6cm]{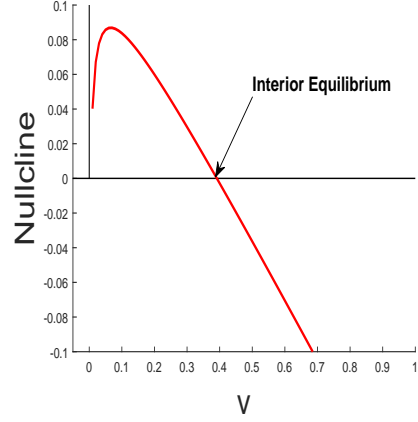}
        \caption{The parameters used are $a_1=0.4,a_2=0.5,b_1=1,b_2=0.6,c_1=0.3,c_2=0.8,p=0.6,q=0.98$.}
        \label{one_int}
    \end{subfigure}

    \begin{subfigure}[b]{.475\linewidth}
        \includegraphics[width=6cm, height=6cm]{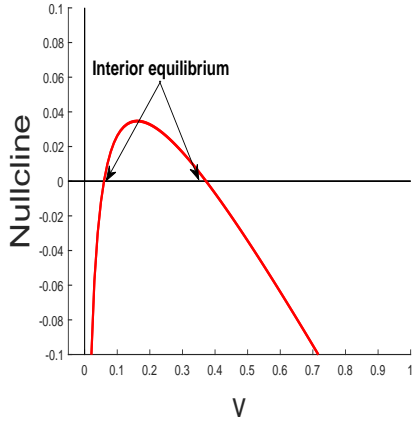}
        \caption{The parameters used are $a_1=0.4,a_2=0.6,b_1=1,b_2=0.6,c_1=0.3,c_2=0.8,p=0.6,q=0.93$. }
        \label{two_int}
    \end{subfigure}
    \hfill
    \begin{subfigure}[b]{.475\linewidth}
        \includegraphics[width=6cm, height=6cm]{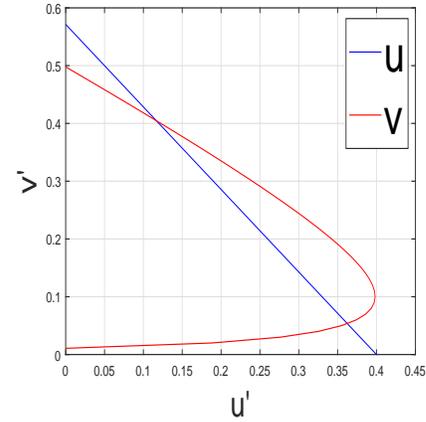}
        \caption{The parameters used are $a_1=0.4,a_2=0.6,b_1=1,b_2=0.6,c_1=0.4,c_2=0.5,p=0.6,q=.9$. }
        \label{complete_nullcline}
    \end{subfigure}
    
    \caption{The figures show the existence of interior equilibria with different parameter conditions as seen in theorem \ref{th:existence} part \ref{int}. Figure \ref{no_int}, \ref{one_int} and \ref{two_int} shows the existence of none, one, and two interior equilibria respectively in system \ref{Eqn:2}. Figure \ref{complete_nullcline} shows the nullclines where the system has three boundary equilibriums and two interior equilibriums. The straight line is the $u$ nullcline and the curved line is the $v$ nullcline.}
    \label{int_exist}
\end{figure*}
 
 \begin{lemma}
 \label{b_1b_2-c_1c_2>0 exists}
 When $b_1b_2-c_1c_2>0$ and $v_{max}^{p-2}=\frac{b_1b_2-c_1c_2}{(1-q)(1-p)}$ then we have:
 \begin{enumerate}
     \item If $(a_2b_1q-c_2a_1)-v_{max}(b_1b_2-c_1c_2)\frac{(2-p)}{(1-p)}<0$ then there does not exist any positive interior equilibrium
     \item If $(a_2b_1q-c_2a_1)-v_{max}(b_1b_2-c_1c_2)\frac{(2-p)}{(1-p)}>0$ then there exists two positive interior equilibria
     \item If $(a_2b_1q-c_2a_1)-v_{max}(b_1b_2-c_1c_2)\frac{(2-p)}{(1-p)}=0$ then there exists one unique interior equilibrium
 \end{enumerate}
 \end{lemma}
\begin{proof}

Let $\phi(v)=(a_2b_1q-c_2a_1)+(c_1c_2-b_1b_2)v-b_1(1-q)v^{p-1}$ where $(b_1b_2-c_1c_2)>0$.\\

We try to study the dynamics of the slope of $\phi(v)$ to determine the existence of equilibrium points. \\
We have, $\phi'(v)=(c_1c_2-b_1b_2)-b_1(p-1)(1-q)v^{p-2}$.\\
Now as $(b_1b_2-c_1c_2)>0$ so it $\phi'(v)$ can be of any sign depending on the magnitude of $v$.\\
Let us assume there exists an extrema $v_{max}$ then we have, $\phi'(v_{max})=0$.\\
Solving for the extrema we get, $v_{max}^{p-2}=\frac{b_1b_2-c_1c_2}{b_1(1-q)(1-p)}$. As $b_1b_2-c_1c_2>0$ and $0<p,q,1$ thus $v_{max}$ is positive. \\
To study the nature of the extrema we find the higher derivative which gives: $\phi''(v_{max})=-b_1(1-q)(1-p)(2-p)v_{max}^{p-3}$. As $0<p,q<1$ so $\phi''(v_{max})$ is strictly negative for any value of $p$ and $q$ . Thus at $v_{max}$ we have a maxima. Now as the parameters are fixed so $v_{max}$ is a unique positive root. So, $v_{max}$ is a global maximum.\\
Then the maximum value of the function is \\
$\phi(v_{max})= (a_2b_1q-c_2a_1)+(c_1c_2-b_1b_2)v_{max}-b_1(1-q)v_{max}^{p-1}$\\ 
=$(a_2b_1q-c_2a_1)-v(b_1b_2-c_1c_2)\frac{(2-p)}{(1-p)}$\\
We see $\phi(v) \to  -\infty$ and when  $v\to 0$ and $v\to \infty$. So the curve is downward facing as seen in figure \ref{int_exist}.\\
So if, $\phi(v_{max})<0$ then it is obvious that there does not exist any positive root of $\phi(v)$.
If $\phi(v_{max})=0$ then $v_{max}$ is the only root of $\phi(v)$.
As $v_{max}$ is the only maxima so if $\phi(v_{max})>0$ then we can have two distinct positive roots of $\phi(v)$. Hence the number of equilibria are determined for the polynomial $\phi(v)$ when $b_1b_2-c_1c_2>0$.
\end{proof}

According to lemma \ref{Boundary_v}, \ref{b_1b_2-c_1c_2<0 exists} and \ref{b_1b_2-c_1c_2>0 exists} we  know the existence of equilibria of system (\ref{Eqn:2}) and can be summarized to theorem \ref{th:existence} given by:

\begin{theorem} \hfill\\
\label{th:existence}
\begin{enumerate}
    \item System (\ref{Eqn:2}) has a trivial equilibrium $E_0(0,0)$.
    \item System (\ref{Eqn:2}) has two types of boundary equilibrium {$E_u(a_1/b_1,0)$} and $E_v(0,\Bar{v})$. The number of equilibrum of the form $E_v(0,\Bar{v})$ can be determined by:
    \begin{enumerate}
     \item If $q < \frac{b_2v_\phi(2-p)}{a_2(1-p)}$ and $v_{\phi}^{p-2}=\frac{b_2}{(1-q)(1-p)}$ then there does not exist any boundary equilibrium of the form $E_v(0,\Bar{v})$.
     \item If $q = \frac{b_2v_\phi(2-p)}{a_2(1-p)}$ and $v_{\phi}^{p-2}=\frac{b_2}{(1-q)(1-p)}$ then there exists an unique boundary equilibrium $E_v(0,\Bar{v})$.
     \item If $q > \frac{b_2v_\phi(2-p)}{a_2(1-p)}$ and $v_{\phi}^{p-2}=\frac{b_2}{(1-q)(1-p)}$ then there exists two boundary equilibrium $E_{v_1}(0,\Bar{v_1})$ and $E_{v_2}(0,\Bar{v_2})$.
        
    \end{enumerate}
    \item If  $(b_1b_2-c_1c_2) \leq 0$ then system (\ref{Eqn:2}) has an unique equilibrium.
    \item \label{int}If $(b_1b_2-c_1c_2) > 0$ and $v_{max}^{p-2}=\frac{b_1b_2-c_1c_2}{(1-q)(1-p)}$ then
        \begin{enumerate}
            \item If $(a_2b_1q-c_2a_1)-v_{max}(b_1b_2-c_1c_2)\frac{(2-p)}{(1-p)}<0$ then system (\ref{Eqn:2}) has no equilibrium.
            \item If $(a_2b_1q-c_2a_1)-v_{max}(b_1b_2-c_1c_2)\frac{(2-p)}{(1-p)}=0$ then system (\ref{Eqn:2}) has an unique equilibrium.
            \item If $(a_2b_1q-c_2a_1)-v_{max}(b_1b_2-c_1c_2)\frac{(2-p)}{(1-p)}>0$ then system (\ref{Eqn:2}) has two different equilibria.
        \end{enumerate}
\end{enumerate}
\end{theorem}

\subsection{Stability of the equilibria.}\hfill\\

In this subsection, we study the dynamics of the system (\ref{Eqn:2}) in the neighborhood of each equilibrium. We study the stability of each equilibrium using the Jacobian matrix $J(u,v)$ of the system (\ref{Eqn:2}) given by,\\

$\begin{bmatrix}
a_1-2b_1u-c_1v & -c_1u \\
-c_2v & a_2q-2b_2v-c_2u-p(1-q)v^{p-1}
\end{bmatrix}$

The trivial equilibrium $E_0(0,0)$ has two eigenvalues where $\lambda_1=a_1$ which is always positive. So, $E_0(0,0)$ is always unstable. Next, we investigate the stability of the boundary and the interior equilibria. The analysis of the boundary equilibrium $E_u(a_{1}/b_1,0)$, is not possible via linear stability analysis, due to the lack of differentiability of system (\ref{Eqn:2}) at $u=0$. Nonetheless, the following result can be provided,

\begin{lemma}
\label{interior_stab}
Consider the boundary equilibrium $E_u(a_{1}/b_1,0)$. This is locally stable. That is there exists certain data that is attracted to this equilibrium in finite time.
\end{lemma}
\begin{proof}
The proof follows via methods in \cite{KRM20, P19}.
\end{proof}

\begin{lemma}
\label{stab_interior}
Let $v_{\phi}^{p-2}=\frac{b_2}{(1-q)(1-p)}$ then
\begin{enumerate}
        \item If $q = \frac{b_2v_\phi(2-p)}{a_2(1-p)}$ then there exists an unique boundary equilibrium $E_v(0,v_\phi)$ and it is a saddle node.
     \item If $q > \frac{b_2v_\phi(2-p)}{a_2(1-p)}$ then there exists two boundary equilibrium $E_{v_1}(0,v_1)$ and $E_{v_2}(0,v_2)$ and where $v_1<v_2$ and $E_{v_1}$ is a saddle point while $E_{v_2}$ is a stable node if $v_2>\frac{a_1}{c_1}$
\end{enumerate}
\end{lemma}

\begin{lemma}
\label{unstab_interior}
If $(b_1b_2-c_1c_2)\leq 0$ then the unique interior point is always an unstable point. 
\end{lemma}

\begin{lemma}
\label{stable_equi}
If $b_1b_2-c_1c_2>0$ then we two positive interior point where $E(u^*,v^*)$ is a stable if 
$q<\frac{1}{a_2b_1(1-p)}\min(a_1+a_1c_2(1-p)+((b_2-c_1)+(b_1b_2-c_1c_2)(1-p))v^*,a_1c_2(1-p)+(b_1b_2-c_1c_2)\frac{2-p}{b_1}v^*)$ \\
and $q>\frac{1}{a_2b_1(1-p)}((b_1b_2-c_1c_2)(2-p)v_{max}+c_2a_1(1-p))$ where $v_{max}^{p-2}=\frac{b_1b_2-c_1c_2}{(1-q)(1-p)}$ else it is a saddle point. 
\end{lemma}

According to Lemma \ref{interior_stab}, \ref{unstab_interior} and \ref{stable_equi} we have an understanding of the stability of the equilibria of the system (\ref{Eqn:2}). The results can be summarized into a theorem given by:\\

\begin{theorem}
\label{th:stability}
The study of stability of the system (\ref{Eqn:2}) shows\\
\begin{enumerate}
    \item The trivial equilibrium $E_0(0,0)$ is always unstable.
    \item The boundary equilibrium $E_u(\frac{a_1}{b_1},0)$ is stable when 
    \item The boundary equilibrium $E_v(0,\Bar{v})$ is stable when let $v_{\phi}^{p-2}=\frac{b_2}{(1-q)(1-p)}$ then
    \begin{itemize}
        \item If $q = \frac{b_2v_\phi(2-p)}{a_2(1-p)}$ then there exists an unique boundary equilibrium $E_v(0,v_\phi)$ and it is a saddle node.
     \item If $q > \frac{b_2v_\phi(2-p)}{a_2(1-p)}$ then there exists two boundary equilibrium $E_{v_1}(0,v_1)$ and $E_{v_2}(0,v_2)$ where $v_1<v_2$ and $E_{v_1}$ is a saddle point while $E_{v_2}$ is a stable node if $v_2>\frac{a_1}{c_1}$
    \end{itemize}
    \item \label{no_int_stab} When $b_1b_2-c_1c_2\leq 0$ then the unique interior point is always an unstable point. 
    \item \label{one_int_stab} When $b_1b_2-c_1c_2>0$ and $v_{max}^{p-2}=\frac{b_1b_2-c_1c_2}{(1-q)(1-p)}$ then,
    \begin{itemize}
        \item If $(a_2b_1q-c_2a_1)-v_{max}(b_1b_2-c_1c_2)\frac{(2-p)}{(1-p)}=0$ then the equilibrium point is undergoes Saddle Node Bifurcation
        \item If $(a_2b_1q-c_2a_1)-v_{max}(b_1b_2-c_1c_2)\frac{(2-p)}{(1-p)}>0$ then the equilibrium point is stable if\\
        $\frac{1}{a_2b_1(1-p)}((b_1b_2-c_1c_2)(2-p)v_{max}+c_2a_1(1-p))<q<\frac{1}{a_2b_1(1-p)}\min(a_1+a_1c_2(1-p)+((b_2-c_1)+(b_1b_2-c_1c_2)(1-p))v^*,a_1c_2(1-p)+(b_1b_2-c_1c_2)\frac{2-p}{b_1}v^*)$
    \end{itemize}
\end{enumerate}
\end{theorem}

According to lemma \ref{stable_equi}, we see that when $E_{max}(u_{max},v_{max})$ is the equilibrium and we cross from one side of a plane to another by change of parameter $q$ the system (\ref{Eqn:2}) changes from no equilibrium to two equilibria. So, there can exist a saddle-node bifurcation at $E_{max}(u_{max},v_{max})$ which we study in the next section.

 \begin{figure*}
    \begin{subfigure}[b]{.475\linewidth}
        \includegraphics[width=6cm, height=6cm]{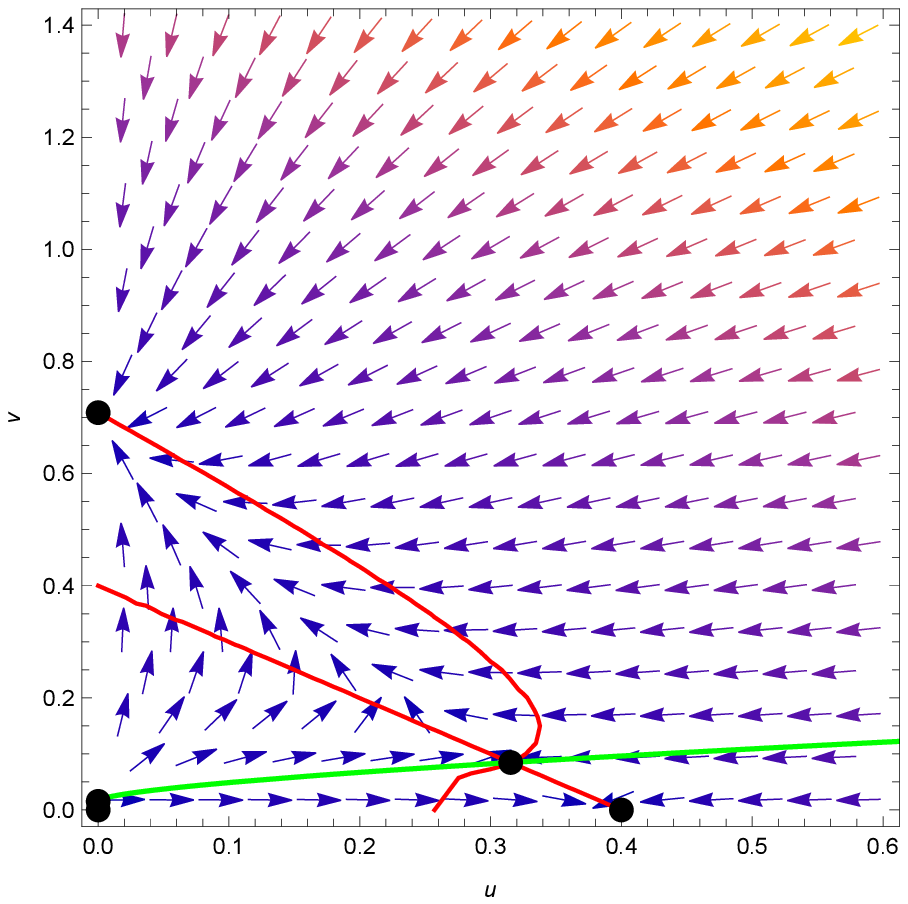}
        \caption{The parameters used are  $a_1=0.5,a_2=0.6,b_1=0.4,b_2=0.6,c_1=0.4,c_2=0.5,p=0.6,q=0.96$. }
        \label{no_int_stable}
    \end{subfigure}
    \hfill
    \begin{subfigure}[b]{.475\linewidth}
        \includegraphics[width=6cm, height=6cm]{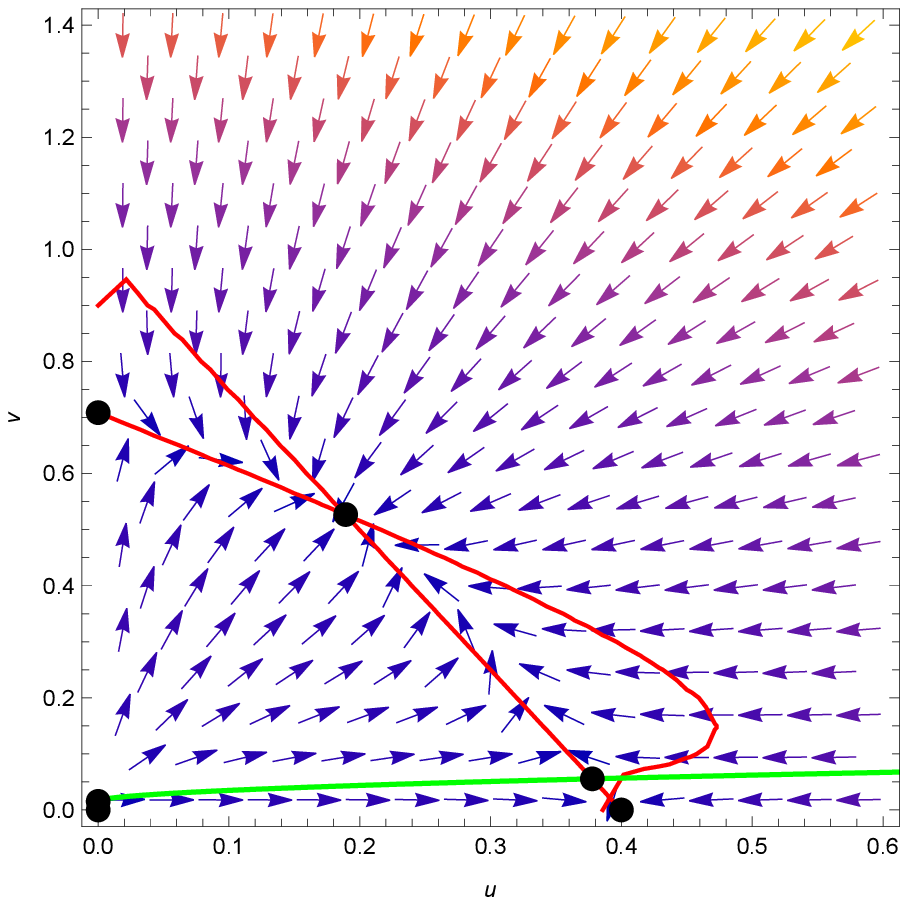}
        \caption{The parameters used are $a_1=0.4,a_2=0.6,b_1=1,b_2=0.6,c_1=0.4,c_2=0.5,p=0.6,q=0.9$.}
        \label{one_int_stable}
    \end{subfigure}
    \caption{The figures show the stability of interior equilibria with different parameter conditions as seen in theorem \ref{th:stability}.The red lines are the nullclines of the system. The green line is the stable manifold passing through saddle equilibrium. The dots are the equilibria for the system. Figure \ref{no_int_stable} shows the instability of the interior equilibrium as seen in  theorem \ref{th:stability}(\ref{no_int_stab}).    Figure \ref{one_int_stable} shows the dynamics of one interior equilibrium stable and one as saddle under the conditions as seen in \ref{th:stability}(\ref{one_int_stab}).}
    \label{int_stable}
\end{figure*}

\subsection{Saddle-node Bifurcation} \hfill\\

When $(b_1b_2-c_1c_2)>0$ then by lemma \ref{b_1b_2-c_1c_2>0 exists} there can be two interior equilibrium points from no equilibrium point when the nullclines crosses $E_{max}(u_{max},v_{max})$ equilibrium point due to change in parameter q.\\
The Jacobian matrix for such a equilibrium point $E_{max}(u_{max},v_{max})$ is
J($E_{max}$)=
$\begin{bmatrix}
-b_1u_{max} & -c_1u_{max}\\ -c_2v_{max} & -b_2v_{max}+(1-p)(1-q)v_{max}^{(p-1)}
\end{bmatrix}$ \\
So, det$J(E_{max})=u_{max}v_{max}((b_1b_2-c_1c_2)-b_1(1-p)(1-q)v_{max}^{(p-2))}$

As $v_{max}$ and $u_{max}$ are not zero then,

det$J(E_{max})=0$ if $v_{max}^{(p-2)}=\frac{b_1b_2-c_1c_2}{b_1(1-p)(1-q}$ 

It is obvious that $\lambda_1=0$ and $\lambda_2=-b_1u_{max}-v_{max}(b_2-(b_1b_2-c_1c_2)/b_1)$ are the eigenvalues of the $J(E_{max})$. So if $-a_1+v_{max}(c_1-b_2+(b_1b_2-c_1c_2)/b_1) \neq 0$ then $\lambda_2 \neq 0.$

So from the conditions, we can obtain that

$SN_1$=$\{(a_1,b_1,b_2,c_1,c_2,p,q) : (b_1b_2-c_1c_2)>0,-a_1+v_{max}(c_1-b_2+(b_1b_2-c_1c_2)/b_1) \neq 0 ,a_1>0,b_1>0,b_2>0,c_1>0,c_2>0,0<p,q<1 \}$ \\
is a saddle-node bifurcation surface. Sotomayor's theorem \cite{perko2013differential} is used to verify the transversality conditions for the occurrence of saddle-node bifurcation of the parameter on the surface $SN_1$. We know that $J(E_{max})$ has one simple zero eigenvalue. If V and W represent the eigenvectors for the zero eigenvalue of the matrix $J(E_{max})$ and $J(E_{max})^T$, respectively, then V and W are,

V=
$\begin{bmatrix}
V_1\\V_2
\end{bmatrix}$= $\begin{bmatrix}
c_1\\-b_1
\end{bmatrix}$,

W=$\begin{bmatrix}
W_1\\W_2
\end{bmatrix}$= $\begin{bmatrix}
c_2v_{max}\\-b_1u_{max}
\end{bmatrix}$.

Furthermore, we have 

$F_q(E_{max}; SN_1)$= $\begin{bmatrix}
0\\a_2v_{max}+v_{max}^{p}
\end{bmatrix}$.

$D^2F(E_{max}; SN_1)(V,V)$=
$\begin{bmatrix}
\frac{\partial^2 F_1}{\partial u^2}  V_1^2 & 2\frac{\partial^2 F_1}{\partial u \partial v}V_1V_2 & \frac{\partial^2F_1}{\partial v^2}V_2^2\\

\frac{\partial^2F_2}{\partial u^2}V_1^2 & 2\frac{\partial^2F_2}{\partial u \partial v}V_1V_2 & \frac{\partial^2F_2}{\partial v^2}V_2^2= 
\end{bmatrix}$ \\
=$\begin{bmatrix}
0\\ \frac{2b_1}{c_1^2}(b_1b_2-c_1c_2)(p-2)
\end{bmatrix}$.

Obviously, V and W satisfy the transversality conditions\\
$W^TF_q(E_{max}; SN_1)=\frac{-b_1u_{max}}{c_1}(a_2+p) \neq 0$,\\
$W^TD^2F(E_{max}; SN_1)(V,V)=\frac{-2b_1^2u_{max}}{c_1^2c_2v_{max}}(b_1b_2-c_1c_2)(p-2) \neq 0$, as $b_1b_2-c_1c_2>0$ and $0<p<1$.\\

for the occurrence of the saddle-node bifurcation at the parameters on the surface $SN_1$. Hence, it can be stated that when the parameters cross from one side of the surface to the other side the number of equilibria of the model changes from zero to two, and the two equilibria which are interior equilibria are hyperbolic saddle and node. The saddle-node bifurcation existence can be numerically seen for a particular parameter set as in figure \ref{saddle_node}.

 \begin{figure*}
    \begin{subfigure}[b]{.475\linewidth}
        \includegraphics[width=6cm, height=6cm]{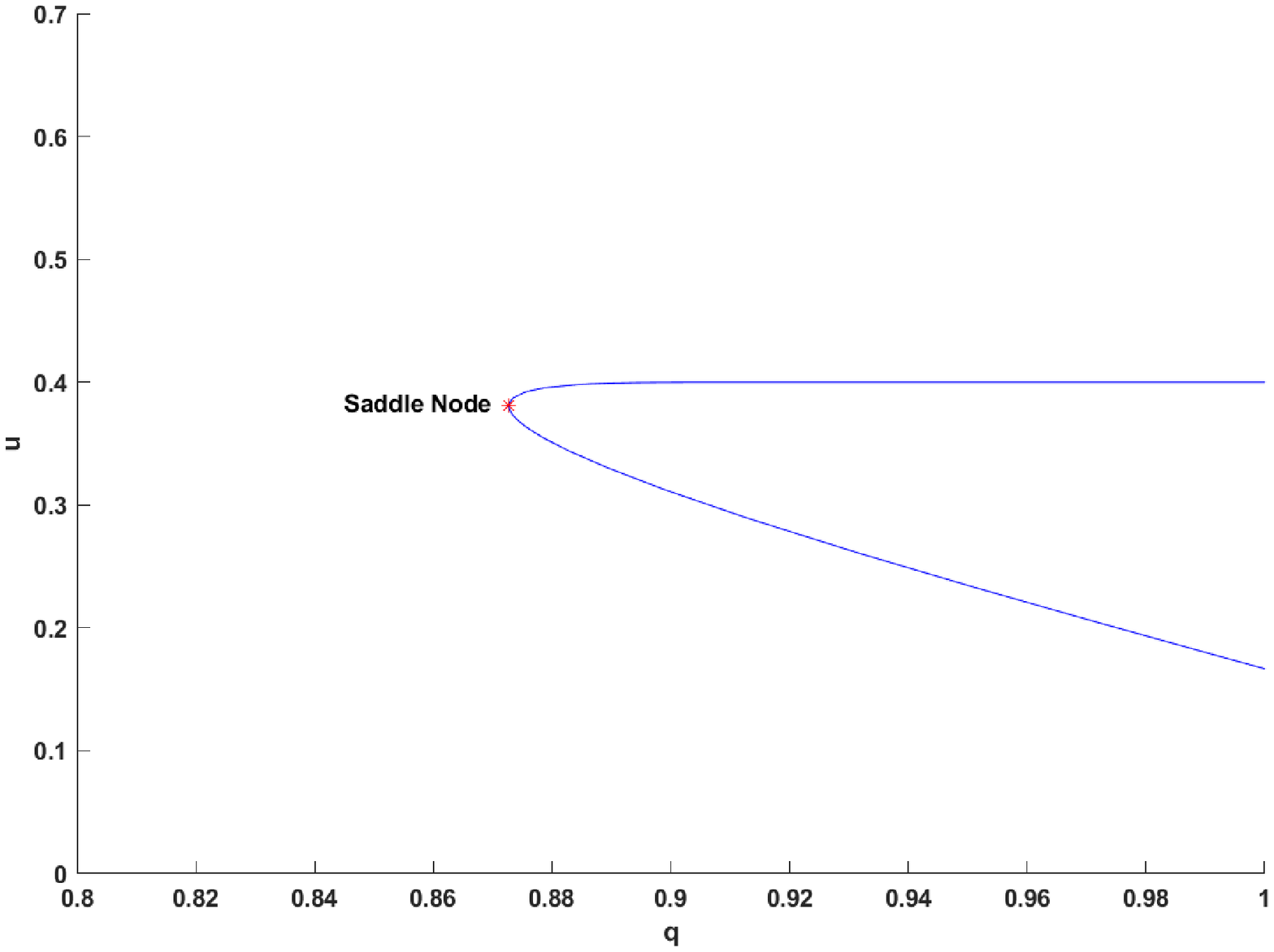}
        \caption{The parameters used are  $a_1=0.4,a_2=0.6,b_1=1,b_2=0.6,c_1=0.3,c_2=0.8,p=0.6,q=0.87272181$. }
        \label{sn_matcont}
    \end{subfigure}
    \hfill
    \begin{subfigure}[b]{.475\linewidth}
        \includegraphics[width=6cm, height=6cm]{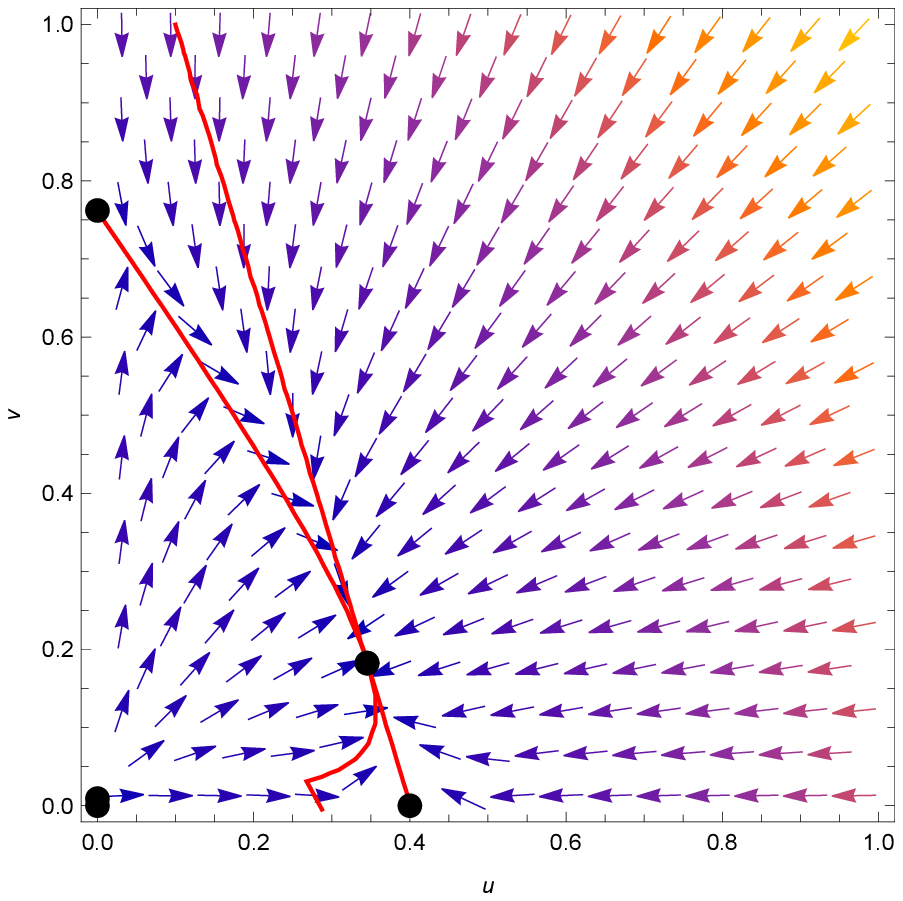}
        \caption{The parameters used are $a_1=0.4,a_2=0.6,b_1=1,b_2=0.6,c_1=0.3,c_2=0.8,p=0.6,q=0.87272181$.}
        \label{sn_nullcline}
    \end{subfigure}

    \begin{subfigure}[b]{.475\linewidth}
        \includegraphics[width=6cm, height=6cm]{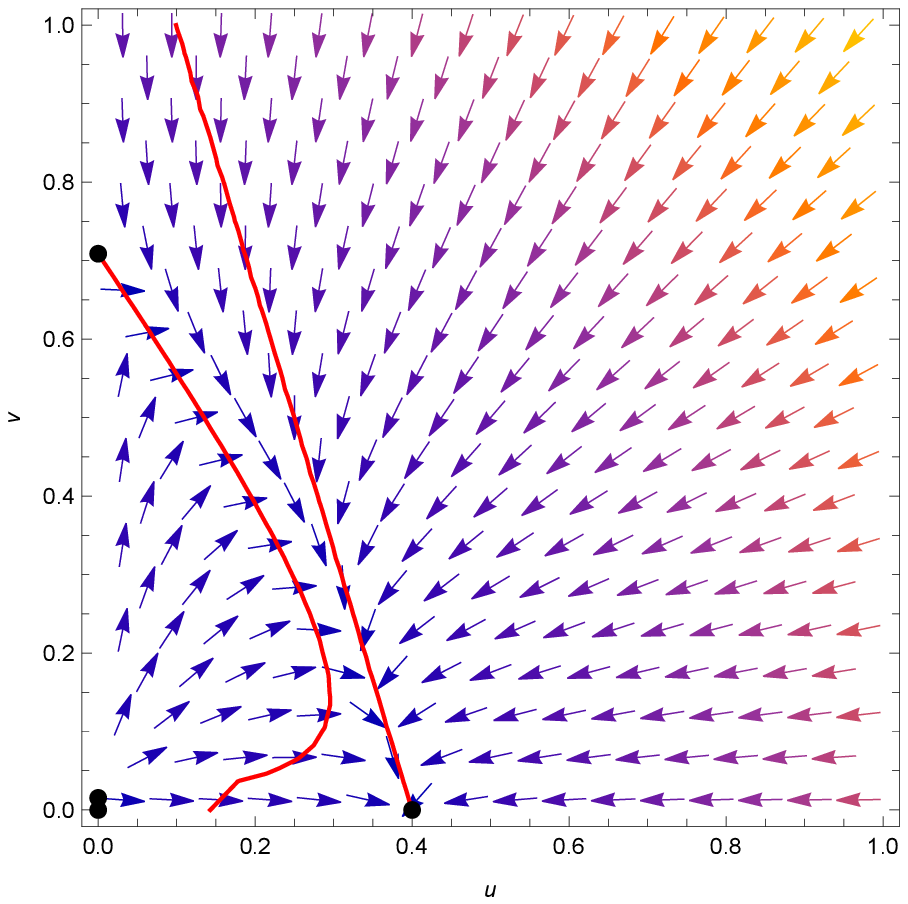}
        \caption{The parameters used are $a_1=0.4,a_2=0.6,b_1=1,b_2=0.6,c_1=0.3,c_2=0.8,p=0.6,q=0.86$. }
        \label{sn_nullcline_no_int}
    \end{subfigure}
    \hfill
    \begin{subfigure}[b]{.475\linewidth}
        \includegraphics[width=6cm, height=6cm]{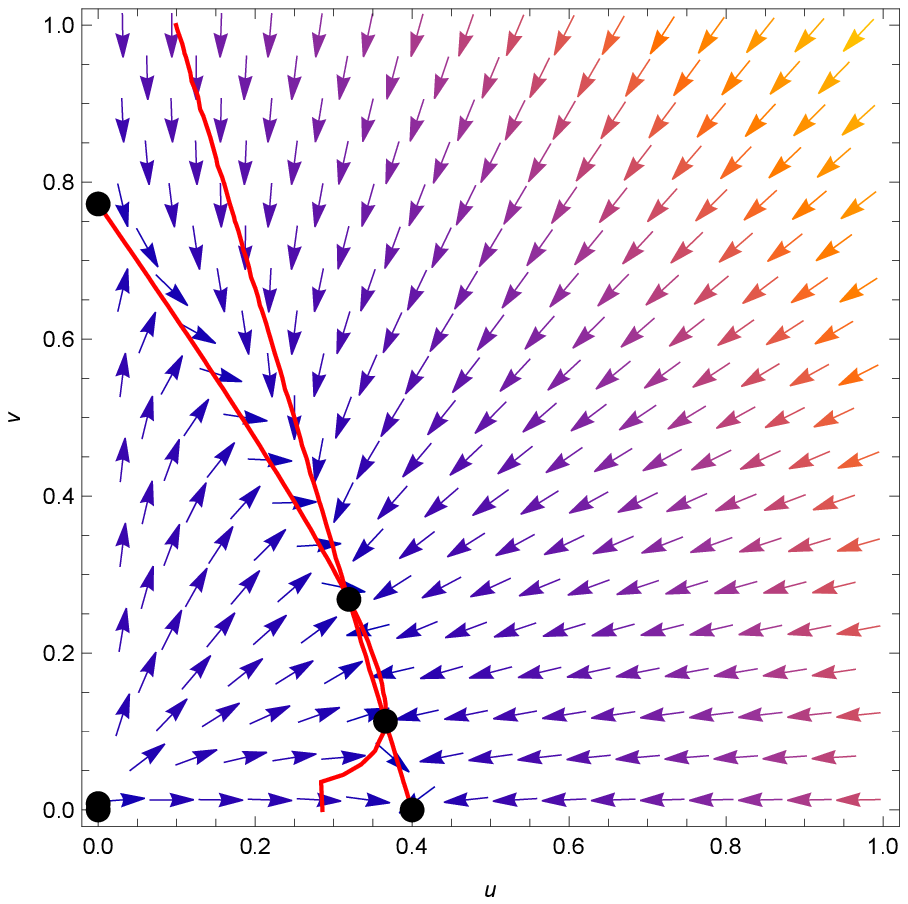}
        \caption{The parameters used are $a_1=0.4,a_2=0.6,b_1=1,b_2=0.6,c_1=0.3,c_2=0.8,p=0.6,q=0.92.$ }
        \label{sn_nullcline_two_int}
    \end{subfigure}
    
    \caption{The figures show the existence of Saddle Node bifurcation with change in $q$ parameter. Figure \ref{sn_matcont} and \ref{sn_nullcline} show the bifurcation diagram and the nullcline dynamics at the bifurcation threshold, respectively.Fig \ref{sn_nullcline_no_int} and fig \ref{sn_nullcline_two_int} show the existence of no and two positive equilibria with decreasing and increasing the $q$ parameter respectively from the bifurcation threshold for the same parameter set.}
    \label{saddle_node}
\end{figure*}

\subsection{Pitchfork Bifurcation}

\begin{theorem}
\label{pitchfork_th}
Consider system (\ref{Eqn:2}), when the boundary equilibria $E_v(0,\Bar{v})$ satisfies the nullcline equation $qa_2-b_2\Bar{v}-(1-q)\Bar{v}^{p-1}=0$ then a pitchfork bifurcation occurs as $q \to q^*$, where,\\
\begin{equation*}
    q^*=\frac{\Bar{v}}{a_2}\left(\frac{c_2(b_1b_2-c_1c_2)}{(1-p)}+b_2\right) \nonumber
\end{equation*}

\end{theorem}

\begin{proof}
We obtain the bifurcation parameter by studying the gradient of the nullclines.\\
The nullclines from system (\ref{Eqn:2}) are, 

 \begin{equation*}
  \begin{split}
    u &= f(v)=\frac{a_1}{b_1}-\frac{c_1}{b_1}v; \\   \nonumber 
    u &= g(v)=\frac{qa_2}{c_2}=\frac{b_2}{c_2}v-\frac{(1-q)}{c_2}v^{p-1} 
      \end{split} 
  \end{equation*}
    The slope of the nullclines are to be determined at the equilibria  $E_v(0,\Bar{v})$. The explicit form of $\Bar{v}$ can be derived from the v-nullclines as it is a root of the equation $qa_2-b_2\Bar{v}-(1-q)\Bar{v}^{p-1}=0$. The gradiant of the nullclines at $E_v(0,\Bar{v})$ is given by $\frac{df}{dv}|_{v=\Bar{v}}=-\frac{c_1}{b_1}$ and $\frac{dg}{dv}|_{v=\Bar{v}}=-\frac{b_2}{c_2}+\frac{(1-q)(1-p)}{c_2}\Bar{v}^{p-2}$. When pitchfork bifurcation takes place, the interior equilibria collide with the boundary equilibrium and the slopes of the nullclines are equal. Thus, $-\frac{c_1}{b_1}=-\frac{b_2}{c_2}+\frac{(1-q)(1-p)}{c_2}\Bar{v}^{p-2}$. Simplifying the equality with the fact that $qa_2-b_2\Bar{v}-(1-q)\Bar{v}^{p-1}=0$ yields the following result,
    \begin{equation*}
    q^*=\frac{\Bar{v}}{a_2}\left(\frac{c_2(b_1b_2-c_1c_2)}{(1-p)}+b_2\right) \nonumber
\end{equation*}.\\
As $q$ decreased to $q^*$ the $g(v)$ nullcline moves downward, and the three interior equilibriums come closer together. This follows by the shape of the nullclines, under the parametric restriction enforced. Since the two slopes of prey and predator nullclines at x = 0 are chosen to be the same, by continuity, the three equilibriums merge into one as $q \to q^*$. Now as $q$ is further decreased, the $g(v)$ nullcline is completely below the prey nullcline, $g(x) < f(x), \forall x$, and there is a single boundary equilibrium. Thus by definition, pitchfork bifurcation has occurred.\\
For a set parameter space, pitchfork bifurcation can be seen in figure \ref{pitchfork}. It can be seen that at two different values of $q$, saddle-node bifurcation and pitchfork bifurcation takes place.
 \begin{figure*}
    \begin{subfigure}[b]{.475\linewidth}
        \includegraphics[width=6cm, height=6cm]{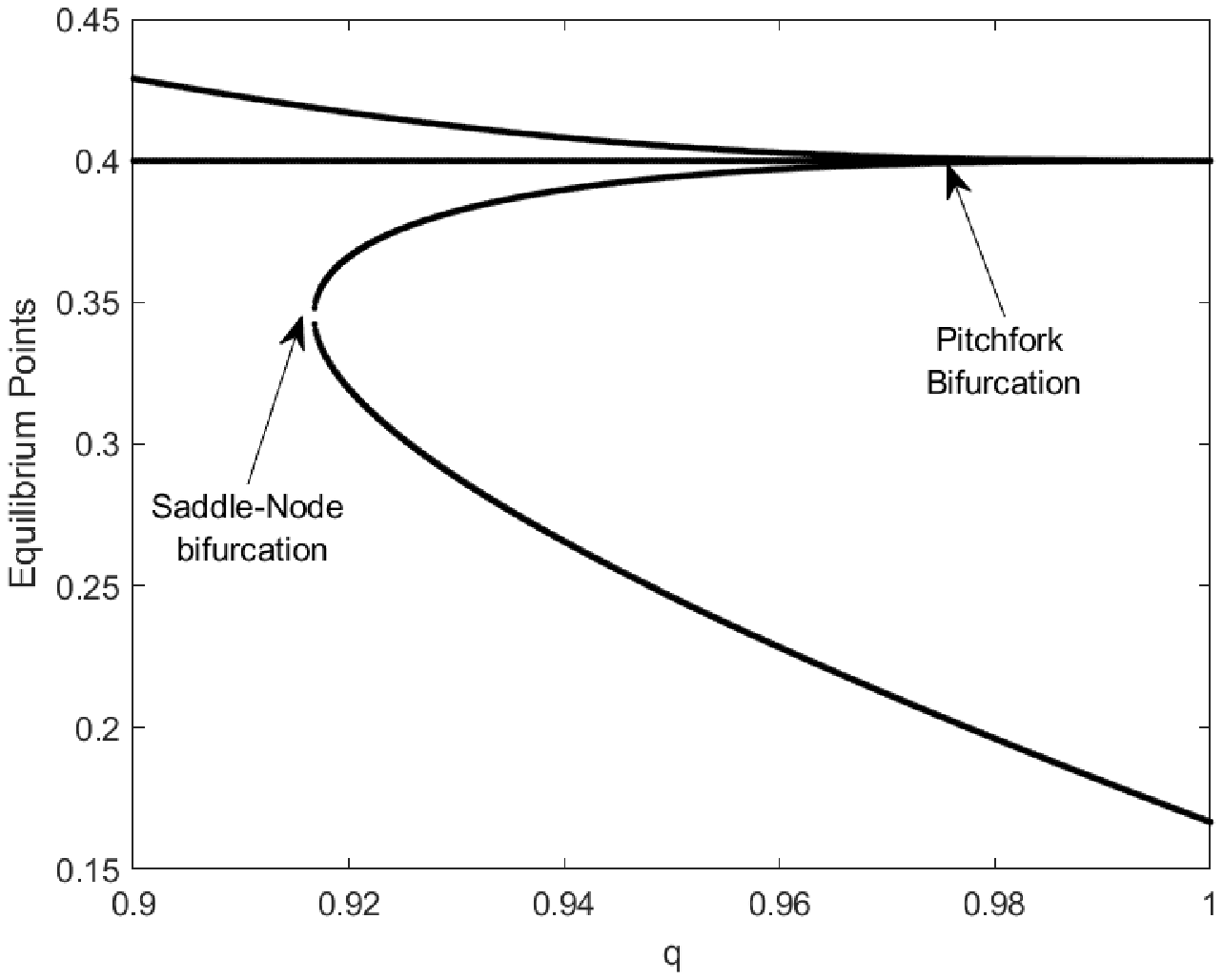}
        \caption{}
        \label{both_bif}
    \end{subfigure}
    \hfill
    \begin{subfigure}[b]{.475\linewidth}
        \includegraphics[width=6cm, height=6cm]{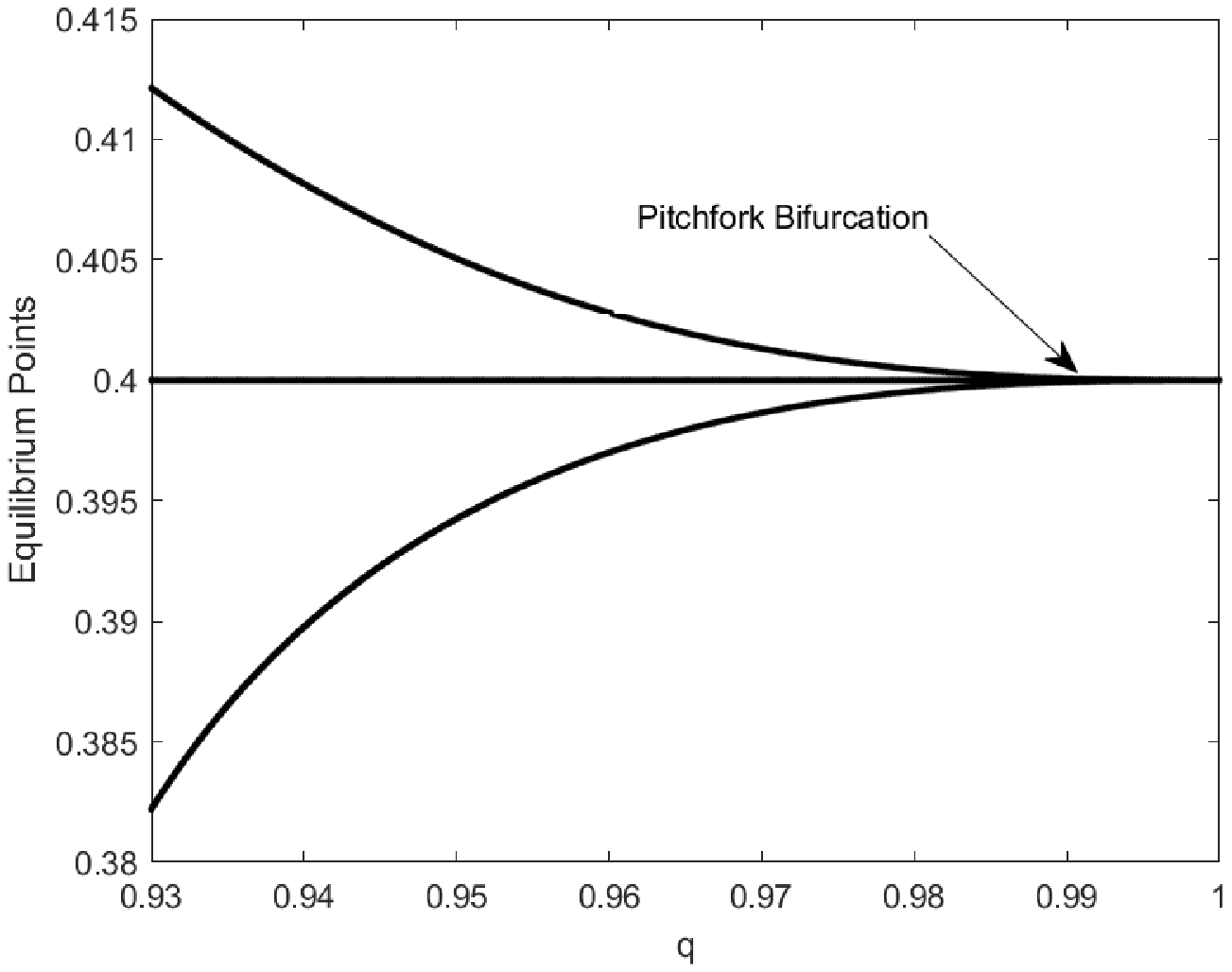}
         \caption{}
         \label{pitchfork}
    \end{subfigure}
    \caption{The parameter set used for the figures is $a_1=0.4, a_2=0.6, b_1=1, b_2=.6, c_1=.3, c_2=.8, p=0.6$. The figures show the equilibria of the system corresponding to changes in parameter $q$. The change in q shows both bifurcations taking place in figure \ref{both_bif} and the zoomed-in figure shows the pitchfork bifurcation in figure \ref{pitchfork}.} 
    \label{pitch_bif}
\end{figure*}

\end{proof}

\section{Comparison with Classical Competition Model}

The system \ref{Eqn:2} can be trivially transformed to a classical competition system \ref{Eqn:1} if $q=1$. We study the different cases of how the dynamics of the system change from the classical case with a change of parameter $p$ and $q$ in the range $(0,1)$.
\subsection{Competitive Exclusion}
\hfill\\
    The classical system shows that when $\dfrac{a_{1}}{a_{2}}<\min\left\lbrace\dfrac{b_{1}}{c_{2}},\dfrac{c_{1}}{b_{2}}\right\rbrace$ then $E_v = (0,\frac{a_2}{b_2})$ is globally asymptotically stable. So, for any parameter set we choose with $q=1$ and satisfying the condition then population $v$ wins, and population $u$ goes to extinction.

If we choose a $q$ very close to 1 then the dynamics change. The parameter chosen satisfying the condition can be seen in figure \ref{comp_exclusion_classic} which satisfies the condition and $E_v = (0,\frac{a_2}{b_2})$ is globally asymptotically stable. With the choice of $q=0.91$ and $p=0.6$, we see that $E_v = (0,\frac{a_2}{b_2})$ becomes unstable and we get two stable equilibria i.e. u boundary equilibrium and an interior equilibrium on the two sides of the stable manifold of an interior saddle equilibrium as seen in figure \ref{comp_exclusion_non_classic}. So, for any initial data below the stable manifold $u$ boundary equilibrium is stable where v population goes to extinction while above the stable manifold coexistence takes place where $E^*$ equilibrium becomes stable. Thus, the $u$ population is always stable in the system and does not go extinct which is the case in the classical model.

\subsection{Weak Competition} \hfill\\
    The classical competition model provides evidence of weak competition when under parametric restrictions coexistence of two species is possible and stable over time. The parametric restriction to have a stable coexistence equilibrium is $\frac{c_2}{b_1}<\frac{a_2}{a_1}<\frac{b_2}{c_1}$ and $b_1b_2-c_1c_2>0$.In the classical competition model (\ref{Eqn:1}) we can get one interior equilibrium that can be asymptotically stable under these restrictions which in system \ref{Eqn:2} can be achieved by satisfying these parametric restrictions and with $q=1$ as seen in \ref{weak_comp_classic}.

In system \ref{Eqn:2} if we use $0<q<1$ then we can get two interior equilibria with one stable point and the other as a saddle point. Due to the saddle, we can have bi-stability where the stable manifold divides the invariant subspace into two regions, depending on the initial condition the populations can either go to the interior equilibrium or they can move to $u$ boundary equilibrium as seen in figure \ref{weak_comp_non_classic}.

 \begin{figure*}
    \begin{subfigure}[b]{.475\linewidth}
        \includegraphics[width=6cm, height=6cm]{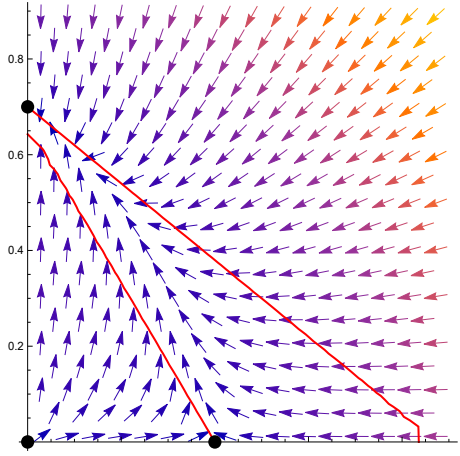}
        \caption{The parameters used are  $a_1=0.4,a_2=0.7,b_1=1,b_2=1,c_1=0.6,c_2=0.8$. }
        \label{comp_exclusion_classic}
    \end{subfigure}
    \hfill
    \begin{subfigure}[b]{.475\linewidth}
        \includegraphics[width=6cm, height=6cm]{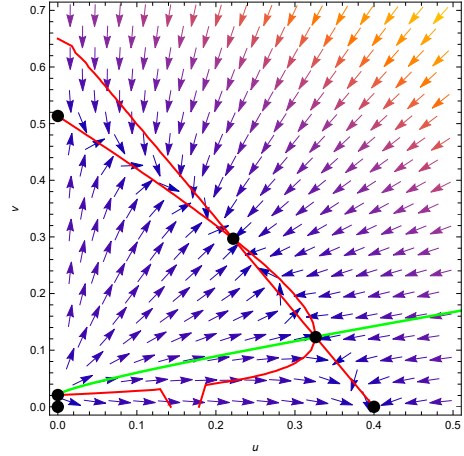}
        \caption{The parameters used are $a_1=0.4,a_2=0.7,b_1=1,b_2=1,c_1=0.6,c_2=0.8,p=0.6,q=0.91$.}
        \label{comp_exclusion_non_classic}
    \end{subfigure}

    \begin{subfigure}[b]{.475\linewidth}
        \includegraphics[width=6cm, height=6cm]{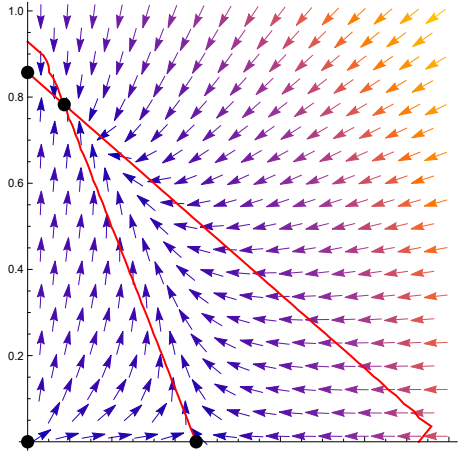}
        \caption{The parameters used are $a_1=0.4,a_2=0.6,b_1=1,b_2=0.7,c_1=0.4,c_2=0.6$. }
        \label{weak_comp_classic}
    \end{subfigure}
    \hfill
    \begin{subfigure}[b]{.475\linewidth}
        \includegraphics[width=6cm, height=6cm]{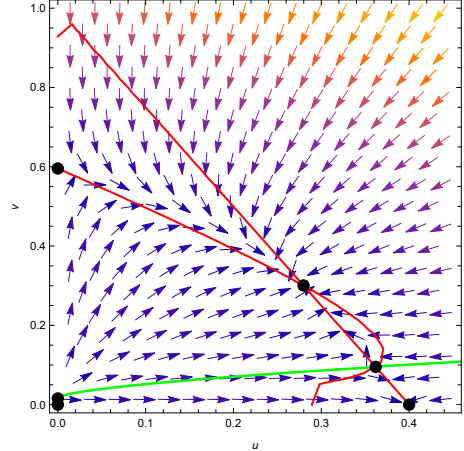}
        \caption{The parameters used are $a_1=0.4,a_2=0.6,b_1=1,b_2=0.7,c_1=0.4,c_2=0.6,p= 0.6,q=0.93$. }
        \label{weak_comp_non_classic}
    \end{subfigure}
    
    \caption{The horizontal and vertical axes are the $u$ and $v$ populations respectively. The red lines are the nullclines and the points are the equilibria. The green lines in \ref{comp_exclusion_non_classic} and \ref{weak_comp_non_classic} are the stable manifolds of the saddle equilibrium. The classical competition cases with $q=1$ are shown in fig \ref{comp_exclusion_classic} and \ref{weak_comp_classic}. The comparison to the classical cases with change $p$ and $q$ are shown in fig \ref{comp_exclusion_non_classic} and \ref{weak_comp_non_classic}. }
    \label{classical_comparison}
\end{figure*}

\begin{figure*}[H]
    \begin{subfigure}[b]{.475\linewidth}
        \includegraphics[width=\linewidth]{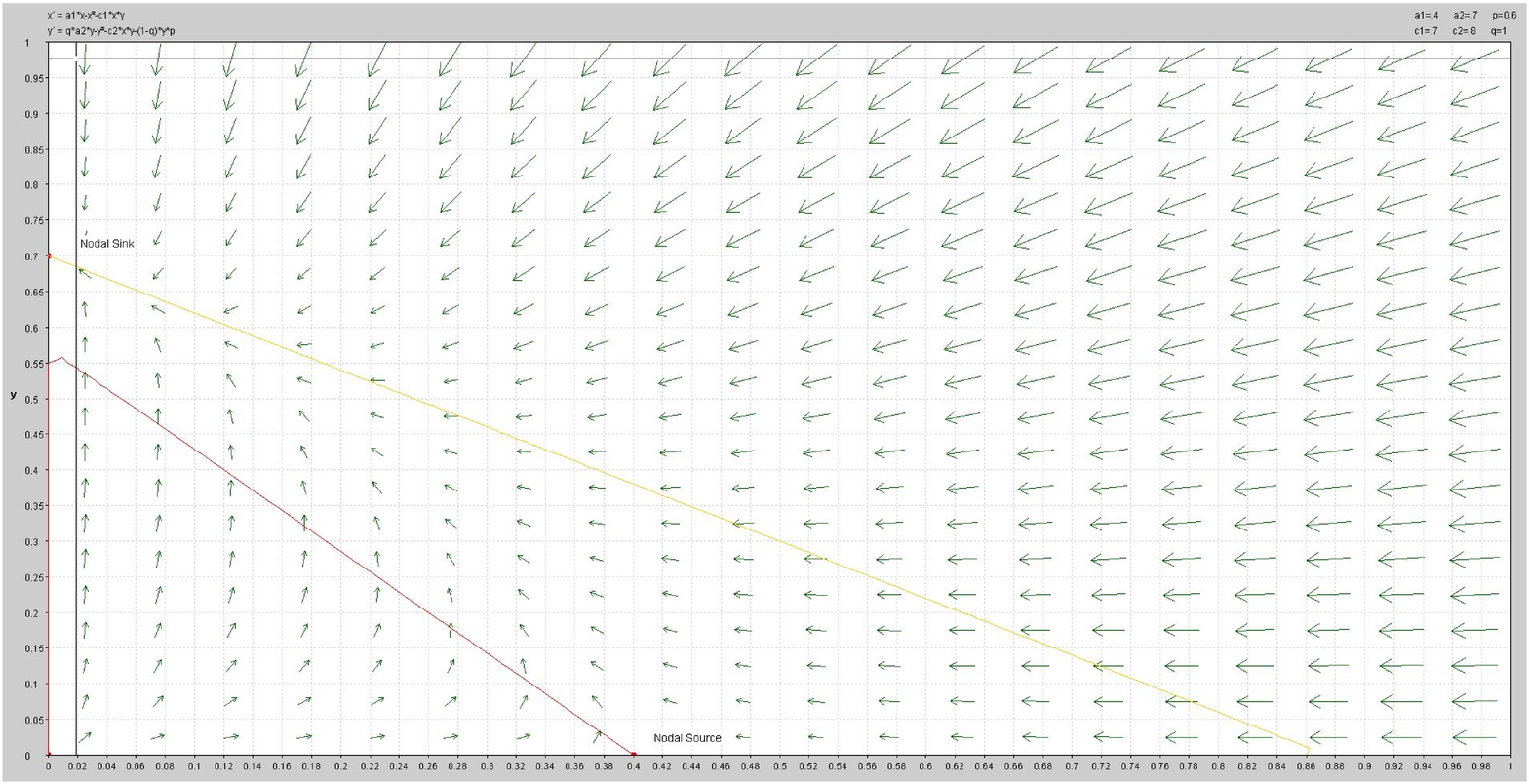}
        \caption{The parameters used for the time series are $a_1=0.1;a_2=0.5;b_1=0.1;b_2=0.4;c_1=0.1;c_2=0.4;$ }
        \label{fig:6MB_BFS}
    \end{subfigure}
    \hfill
    \begin{subfigure}[b]{.475\linewidth}
        \includegraphics[width=\linewidth]{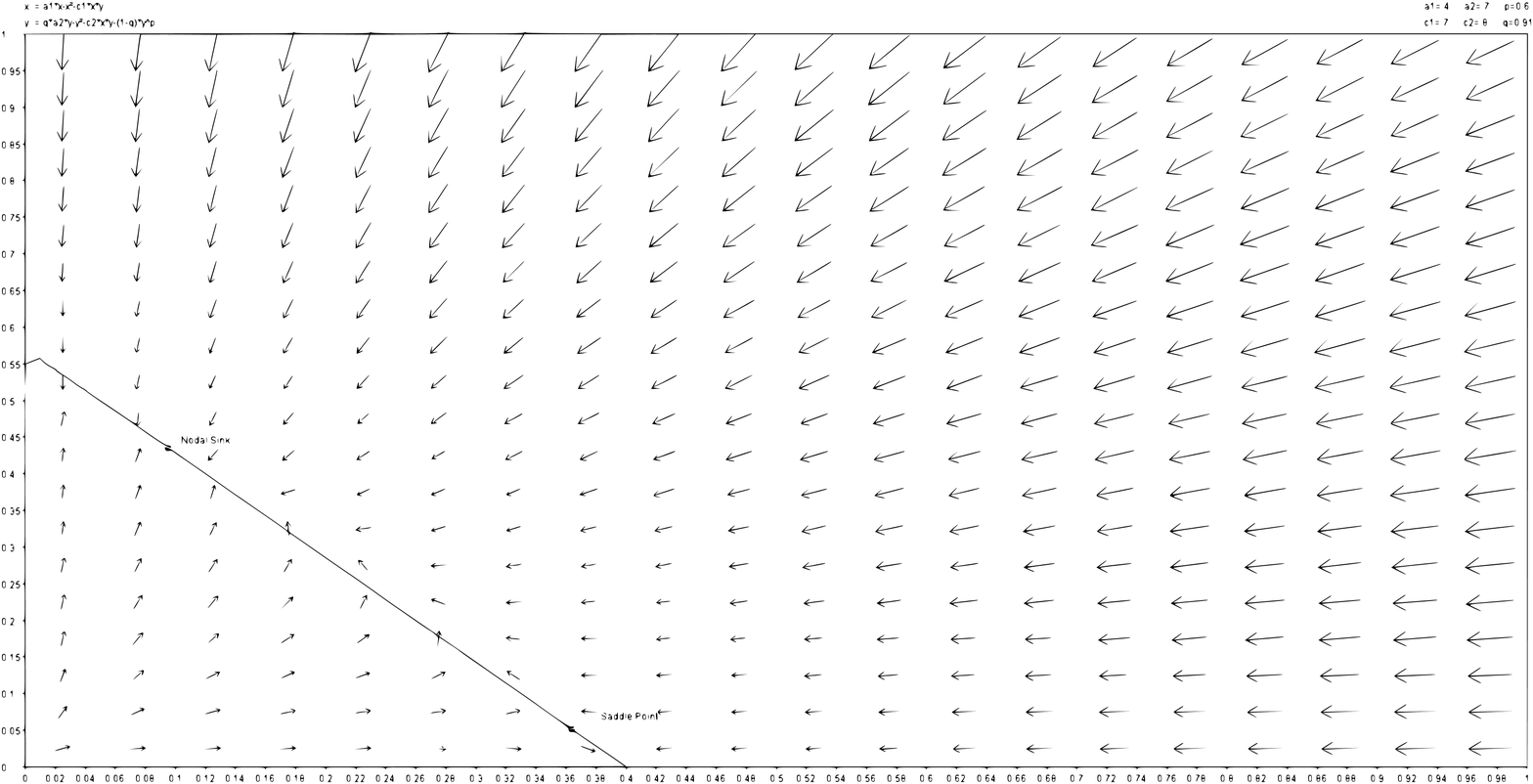}
        \caption{The parameters used for the time series are $a_1=0.1;a_2=0.3;b_1=0.1;b_2=0.4;c_1=0.1;c_2=0.4;$}
        \label{fig:25MB_bfs}
    \end{subfigure}

    \begin{subfigure}[b]{.475\linewidth}
        \includegraphics[width=\linewidth]{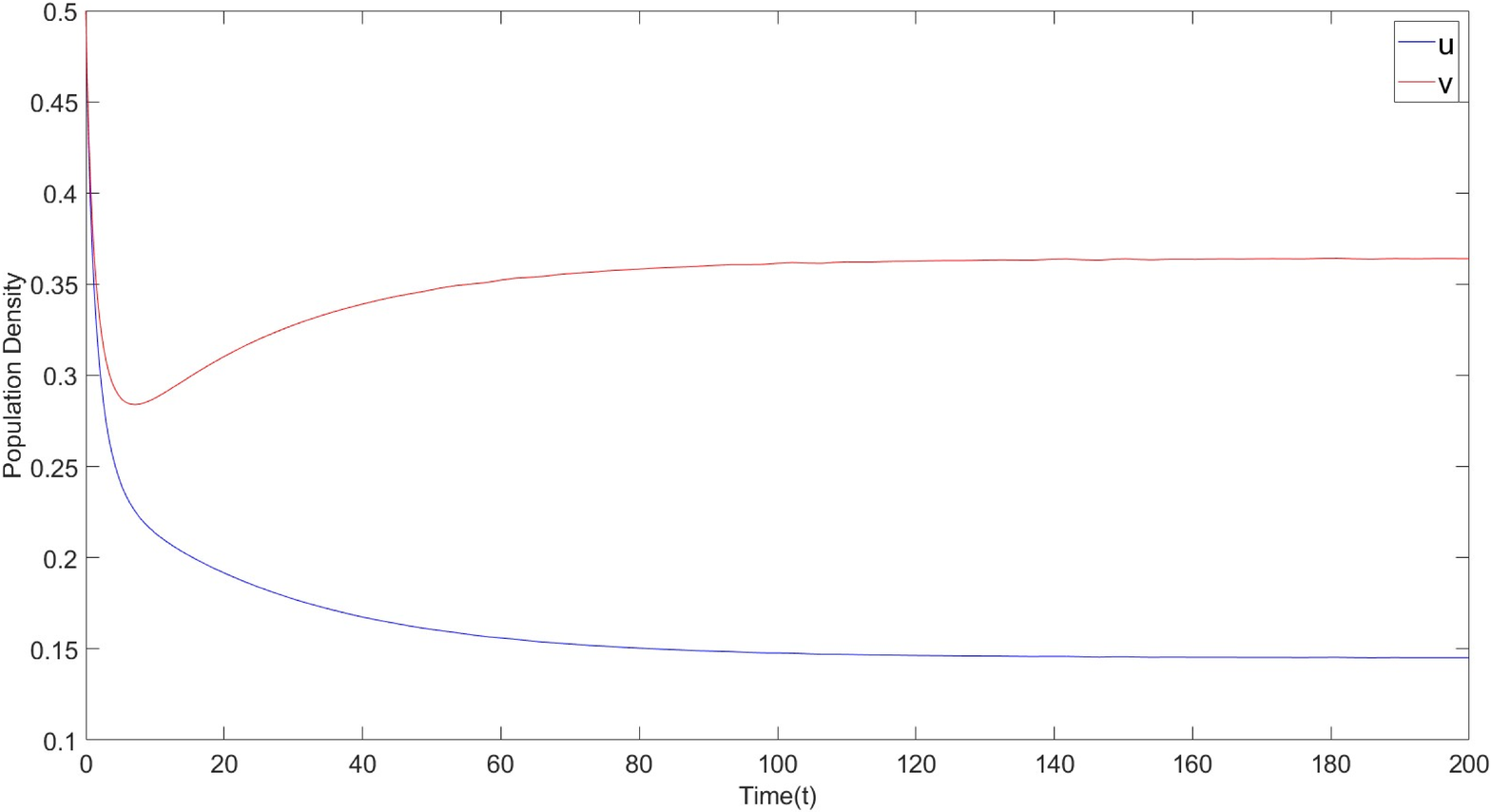}
        \caption{The different equilibrium points with parameters satisfying the stability of the interior equilibrium point. }
        \label{fig:6MB_mm}
    \end{subfigure}
    \hfill
    \begin{subfigure}[b]{.475\linewidth}
        \includegraphics[width=\linewidth]{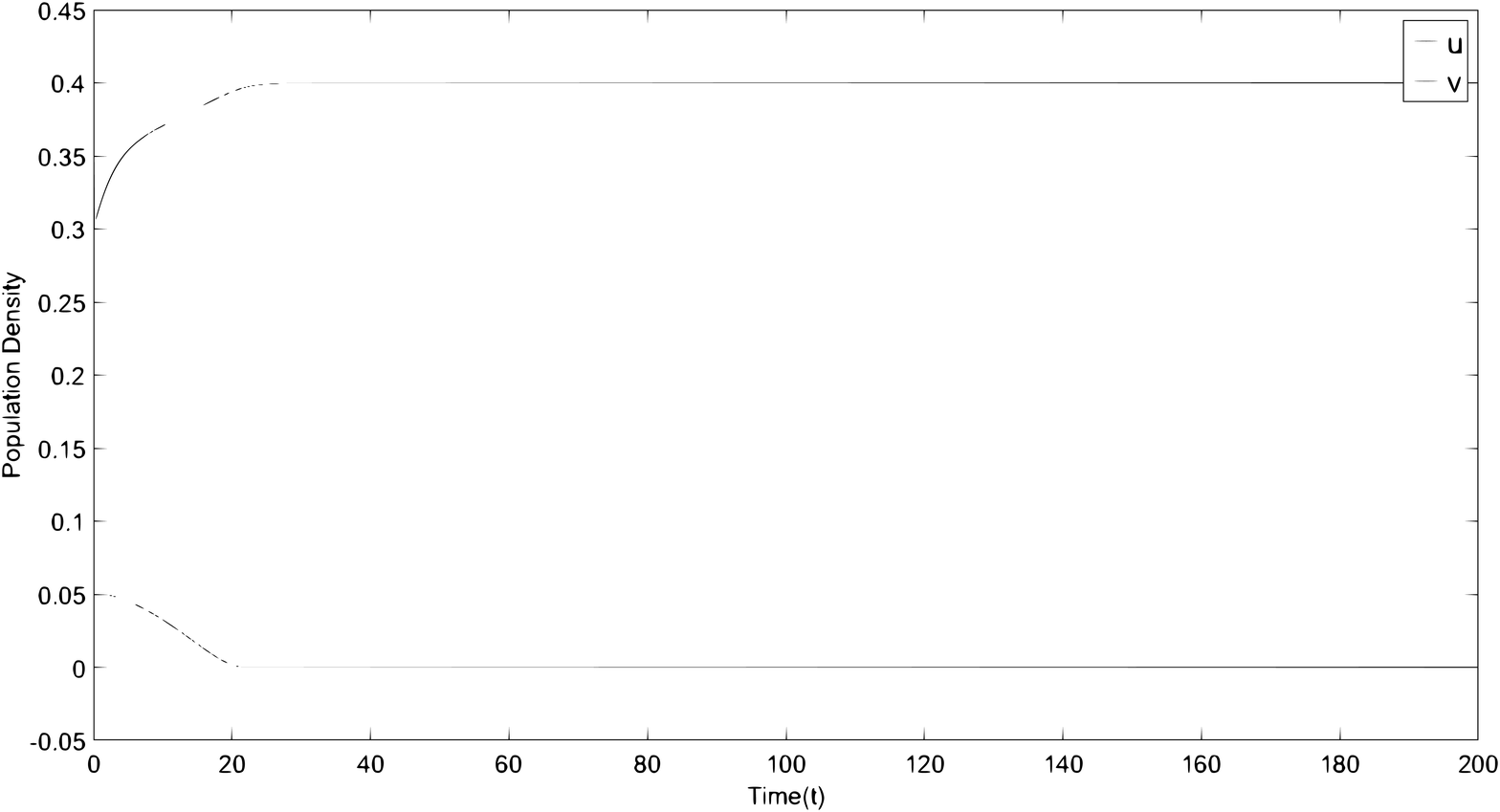}
        \caption{The parameters used for the time series are $a_1=0.1;a_2=0.3;b_1=0.1;b_2=0.4;c_1=0.1;c_2=0.2;$ }
        \label{fig:25MB_mm}
    \end{subfigure}
    \caption{The horizontal and vertical axes are the $u$ and $v$ populations respectively. The red lines are the nullclines and the points are the equilibria. The green lines in \ref{comp_exclusion_non_classic} and \ref{weak_comp_non_classic} are the stable manifolds of the saddle equilibrium. The classical competition cases with $q=1$ are shown in fig \ref{comp_exclusion_classic} and \ref{weak_comp_classic}. The comparison to the classical cases with change $p$ and $q$ are shown in fig \ref{comp_exclusion_non_classic} and \ref{weak_comp_non_classic}. }
    \label{classical_comparison}
\end{figure*}

\newpage
\section{The Spatially Explicit (PDE) Case}

The spatially inhomogeneous problem has been intensely investigated in the past $2$ decades \cite{Cantrell2003,  Chen2020, DeAngelis2016, He2019, He2013a, He2013b, He2016a, He2016b, Hutson2003, Lou2006a, Lou2006b, Liang2012, Li2019, Lam2012, Nagahara2018, Nagahara2020}. The premise here is that $u,v$ do not have resources that are uniformly distributed in space, rather there is a spatially dependent resource function $m(x)$. We consider again a normalized generalization of the classical formulation, where there are $2$ parameters $b$ and $c$ for inter/intra specific kinetics as opposed to $6$ kinetic parameters in \eqref{eq:Ge1} from earlier. The parameter choice $0<p<1$, enables a FTEM in $u$.

\begin{equation}
\label{eq:Ge1pn}
\left\{ \begin{array}{ll}
u_{t} &~ = d_{1}\Delta u +  m(x)u - u^{2} - cu^{p}v ,  \quad  0 < p \leq 1, \\[2ex]
v_{t} &~ =  d_{2}\Delta v + m(x) v -  v^{2} - buv    ,
\end{array}\right.
\end{equation}

\begin{equation}
\label{eq:Ge1bh}
\nabla  u \cdot n = \nabla  v \cdot n  = 0, on \ \partial \Omega\ , \ u(x,0) =  u_{0}(x) > 0, \  v (x,0) =  v_{0}(x) > 0.
\end{equation}
Note, $p=1$, is the classical case. We consider $m$ to be non-negative on $\Omega$ and bounded. We recap a seminal classical result \cite{Dockery1998, Hastings}, which shows that the slower diffuser wins.

\begin{theorem}
Consider \eqref{eq:Ge1pn}-\eqref{eq:Ge1bh}, when $b=c=p=1$, and $d_{1} < d_{2}$, solutions initiating from any positive initial data $(u_{0}(x), v_{0}(x))$ converge uniformly to $(u^{*}(x),0)$. 
\end{theorem}
That is the slower diffuser wins, in the case of equal kinetics.
However, a difference in the interspecific kinetics via FTEM can cause the slower diffuser to \emph{lose}, depending on the initial conditions \cite{parshad2021some}. 

The following conjecture was made (and some numerical evidence provided for) in \cite{parshad2021some} for the system:

\begin{equation}
\label{eq:Ge1pn2}
\left\{ \begin{array}{ll}
u_{t} &~ = d_{1}\Delta u +  m(x)u - u^{2} - cu^{p}v ,  \quad  0 < p \leq 1, \\[2ex]
v_{t} &~ =  d_{2}\Delta v + m(x) v -  v^{2} - bu v^{q}    ,  \quad  0 < q \leq 1.
\end{array}\right.
\end{equation}

\begin{conjecture}[Co-existence when $p=1, 0<q<1$]
\label{con:usgs1}
Consider \eqref{eq:Ge1pn2}-\eqref{eq:Ge1bh} where $\Omega \subset \mathbb{R}$, is a bounded domain, and when $b=c=p=q=1$, $d_{1} < d_{2}$. There exists positive initial data $(u_{0}(x), v_{0}(x))$, for which solutions converge to $(u^{*}(x),0)$, but solutions with the same diffusion coefficients, initiating from the same data, will converge to $(u^{*}(x), v^{*}(x))$ in finite time, for a sufficiently chosen $q \in (0,1)$.
\end{conjecture}

\subsection{Movement Operator}

Consider a species $v$ dispersing over a spatial domain $\Omega$. Its dynamics is typically governed by a diffusion equation,
$v_{t} = \Delta v + f(v)$, where $\Delta v$ represents movement by diffusion, and the other dynamics such as growth, death, competition, depredation are embedded in $f(v)$. Thus one can define a movement operator $L: H^{2}(\Omega) \mapsto L^{2}(\Omega)$, where $ \mathcal{L}(v) = \Delta v$.
We take the following approach to modeling movement. Define, 
\begin{eqnarray*}
\mathcal{L}^{1}(v)  
&=& \mathcal{L}^{1}((1-k)v + k(v)), \ 0<k<1 \  \nonumber \\
&=& (1-k)L(v) + kL^{*}(v) \nonumber \\
&=& \Delta ((1-k)v) + L^{*}(k v),  \nonumber \\
\end{eqnarray*}
 where $\boxed{L^{*}(v) \approx -v^{q}}$, thus the operator $L^{*}(v)$ will play the role of the sub-linear harvesting term, so as to cause finite time extinction. However, this is only affected by a (small) fraction of the population via the fraction $k \in (0,1)$.

\begin{remark}
Thus the movement operator $\mathcal{L}^{1}$, provides a way to formalize the action via which we have a fraction of the population moving via regular diffusion, via the other fraction moving via fast diffusion.
\end{remark}

This leads to the following quasi-linear PDE, representing the interaction of two competing species, with spatially dependent resource function $m(x)(\in L^{\infty}(\Omega))$ and population densities $u(x, t)$ and $v(x, t)$

\begin{equation}
\label{eq:pde_model}
\left\{ \begin{array}{ll}
u_t &= d_1 \nabla \cdot (\nabla u) +  u\Big(m(x)-u-v\Big), \quad x\in \Omega, t>0\\
		v_t &= \nabla \cdot a(x,v,\nabla v)  + v\Big(m(x)-u-v\Big), \quad x\in\Omega,t>0, \\
	 a(x,v,\nabla v) & = \Big( d_2 (1-k)\nabla v  + k |\nabla v|^{p-2}  \nabla v \Big) , \ p \in (1,2], \ 0 \leq k \leq 1 \\
 \nabla u \cdot \eta &=a \cdot \eta=0, \quad x\in \partial \Omega \\
		u(x,0)&=u_0(x), v(x,0)=v_{0} (x), \quad x\in \Omega \\
\end{array}\right.
\end{equation}

where all the parameters $d_i (i=1,2)$ and $k$ are positive and $\Omega \subset \mathbb{R}^n$ is a bounded domain with smooth boundary.


\begin{lemma}
Consider the system \eqref{eq:pde_model}. Then $\nabla a(x,v,\nabla v) \cdot \eta  \iff \nabla v \cdot \eta = 0$.
\end{lemma}

Consider the boundary conditions for $v$,
\[ a \cdot \eta = \Big( d_2 (1-k) \nabla v + k |\nabla v|^{p-2} \cdot \nabla v \Big) \cdot \eta \hspace{.05in} =0 \iff \nabla v \cdot \eta \Big(  d_2 (1-k) + k |\nabla v|^{p-2} \Big)=0.\]
By the positivity of the bracket term, we can consider the Neumann boundary conditions for $v$. Consider the system \eqref{eq:pde_model} updated boundary conditions:
\begin{align}\label{pde_model_bc}
	\begin{split}
		 \nabla u \cdot \eta &=\nabla v \cdot \eta=0, \quad x\in \partial \Omega.
	\end{split}
\end{align}

\begin{theorem}\label{classic_pde}
Consider \eqref{eq:pde_model} in a bounded domain with smooth boundary $\Omega \subset \mathbb{R}$ with $p=2$ and $d_2(1-k) + k < d_1$. Then for any choice of positive initial data $(u_0(x),v_0(x))$, the solution $(u,v)$ converges uniformly to $(0,v^{*})$ as $t\to \infty.$
\end{theorem}

The proof follows via standard techniques and can be found in \cite{dockery1998evolution, ni2011mathematics}.

\begin{definition}[Weak solution]\label{def:weak}
A measurable function $v$ is a local weak sub (super) solution of \eqref{eq:pde_model} in $\Omega_{T}$ if 

\begin{equation}
v \in C_{loc}(0,T;L^{2}_{loc}(\Omega) \cap L^{p}_{loc}(0,T;W^{1,p}_{loc}(\Omega) )
\end{equation}
and for every compact subset $K$ of $\Omega$ and for every sub interval $[t_{1},t_{2}]$ of $(0,T]$

\begin{equation}
\int_{K} v \phi dx \vert^{t_{2}}_{t_{1}}  + \int^{t_{2}}_{t_{1}} \int_{K} \left(   -v \phi_{t}   +  a(x,v,\nabla v) \cdot \nabla \phi    \right) dx d \tau\geq (\leq)  \int^{t_{2}}_{t_{1}} \int_{K}  b(x,\tau,v) \phi dx d \tau
\end{equation}
For all test functions $\phi \in W^{1,2}_{loc}(0,T;L^{2}(K) \cap L^{p}_{loc}(0,T;W^{1,p}_{0}(K) ), \ \phi \geq 0$. 
\end{definition}

\begin{lemma}[Gagliardo–Nirenberg-Sobolev interpolation inequality (GNS)]
Consider functions $\phi \in W^{m,q} (\Omega)$. Then the following interpolation inequality holds,
\begin{equation}\label{GNS}
    ||\phi||_{W^{k,p'} (\Omega)} \le C ||\phi||_{W^{m,q'} (\Omega)}^{\theta} ||\phi||_{\mathcal{L}^q (\Omega)}^{1-\theta},
\end{equation}
for $p',q',q\ge 1, \theta \in [0,1]$ as long as the following is satisfied,
\[ k-\dfrac{n}{p'} \le \theta \Big( m - \dfrac{n}{q'} + \dfrac{n}{q} \Big) -\dfrac{n}{q}. \]

\end{lemma}

We next present the main result of this section.

\begin{theorem}
Consider the system \eqref{eq:pde_model}. Then there exists initial data $(u_{0},v_{0}) \in W^{1,2} (\Omega)$ for which there exist local weak solutions to \eqref{eq:pde_model}.
\end{theorem}


\begin{theorem}\label{FFTE}
Consider the spatially explicit competition model \eqref{eq:pde_model}-\eqref{pde_model_bc} in a bounded domain with smooth boundary $\Omega \subset \mathbb{R}$. There exists some positive initial data $(u_0(x),v_0(x))$ such that for $d_2(1-k) + k < d_1$ and $p=2$, the solution $(u,v) \to (0,v^{*}),$ but for some $p \in(1,2]$ and $k\in (0,1)$, $ (u,v) \to (u^{*},0)$ in $L^{2}(\Omega)$, starting from the same initial data. 
\end{theorem}
\begin{proof}
Consider the system $(\ref{eq:pde_model})$ for $u$. Since the spatially dependent resource function $m(x)( \in \mathcal{L}({\infty}))$ serves as upper-bound (upto some constant $C$) for population density $v(x,t)$, we get
\[ u_t \ge d_1 u_{xx} - u^2-C||m||_{\infty}u.\]
Moreover, on comparison with the standard logistic equation and using the comparison argument, we have $u\ge C_1 e^{-Ct}$ for some constant $C_1$ (depends on $u_0$) and $C$ \cite{Cantrell2003}.

Let's test the PDE $(\ref{eq:pde_model})$ for $v$ against $v$
\[ \int_{\Omega}v v_t =\int_{\Omega}\left( d_2 (1-k) v v_{xx}  + m v^2 -v^2u-v^3 + k v \dfrac{\partial}{\partial x} \Big(|v_x|^{p-2} \cdot v_x \Big) \right).\]
On integrating it on the full domain $\Omega$ and using the homogeneous Neumann boundary conditions $(\ref{pde_model_bc})$, we get
\[ \dfrac{1}{2} \dfrac{d}{dt} ||v||_2^2 + d_2(1-k) ||v_x||_2^2 + C_1 e^{-Ct} ||v||_2^2 + \int_\Omega v^3 + k ||v_x||_p^p \le \int_{\Omega} m v^2 .\]
On using the positivity of $v$ and the fact $m\in \mathcal{L}^{\infty} (\Omega)$, we have
\[ \dfrac{1}{2} \dfrac{d}{dt} ||v||_2^2 + d_2(1-k) ||v_x||_2^2 + k ||v_x||_p^p \le M ||v||_2^2,\]
where $M=||m||_{\infty}.$
Recall the Rellich–Kondrachov theorem \cite{sell2013}, we know 
\[ W^{1,\widetilde{p}} (\Omega) \dhookrightarrow L^{2} (\Omega) \quad \text{if} \quad \frac{n}{\widetilde{p}}-1<\frac{n}{2}.\] Hence, we have
\[ ||v||_2^{2} \le C ||v_x||^p_{p}\]
provided $ p>\frac{2}{3}.$
As $p\in(1,2]$, this embedding hold trivially true. Let us introduce a new parameter $\widehat{p}$ such that $1<\widehat{p}<p$.
Hence, we have
\[ \dfrac{1}{2} \dfrac{d}{dt} ||v||_2^2 + \widetilde{C} \Big\{||v_x||_{\widehat{p}}^{\widehat{p}} + ||v||_p^p\Big\} \le M ||v||_2^2.,\]
where $\widetilde{C}=\min \{d_2 (1-k),\frac{C}{k}\}.$

We will control the term in the braces by using the GNS inequality. Let's compare and note down the corresponding spaces
\[ n=1, \quad  k=0, \quad p'=2, \quad m=1, \quad q'=\widehat{p}, \quad q=p.\]
\begin{equation}\label{GNS_pde}
    ||v||_2 \le C ||v_x||^{\theta}_{\widehat{p}} ||v||^{1-\theta}_{p}
\end{equation}
such that
\[ -\dfrac{1}{2} \le \theta \Big( 1 - \dfrac{1}{\widehat{p}} + \dfrac{1}{p} \Big) - \dfrac{1}{p}\]
On further rearrangement, we have
\begin{equation}\label{GNS_est}
    \theta \ge \dfrac{(2-p)\widehat{p}}{2(\widehat{p}p+\widehat{p}-p).}
\end{equation}
Let's raise the both sides of $(\ref{GNS_pde})$ by $l$, where $l \in (0,2)$
\[ \Big( \int_\Omega v^2 \Big)^{\frac{l}{2}} \le C \Big( \int_\Omega |v_x|^{\widehat{p}} \Big)^{\frac{l \theta}{\widehat{p}}} \Big( \int_\Omega |v|^p \Big)^{\frac{l(1-\theta)}{p}}. \]

Recall the Young's inequality \cite{evans},
\[ ab \le \dfrac{a^r}{r} + \dfrac{b^s}{s}, \]
such that $\frac{1}{r} + \frac{1}{s}=1.$ 
Let's use the Young's inequality for $r = \frac{\widehat{p}}{l \theta}$ and $s=\frac{p}{l(1-\theta)}$. Moreover, $p \in (1,2]$, so we can find a $l \in (0,2)$ such that
\begin{align*}
    \dfrac{l \theta}{\widehat{p}} + \dfrac{l}{p} - \dfrac{l \theta}{p} &=1 \\
    \theta \Big(\dfrac{l}{\widehat{p}} - \dfrac{l}{p} \Big) &=\Big( 1- \dfrac{l}{p} \Big) \\
    \theta &= \Big( \dfrac{\frac{1}{l} - \frac{1}{p}}{\frac{1}{\widehat{p}}-\frac{1}{p}} \Big). \\
    \theta &= \dfrac{\widehat{p}(p-l)}{l(p-\widehat{p})}.
\end{align*}
On comparison with $(\ref{GNS_est})$, we get
\begin{align*}
    \dfrac{(p-l)}{l(p-\widehat{p})} &\ge \dfrac{2-p}{2(\widehat{p}p+\widehat{p}-p)} \\
    \dfrac{p}{l} -1 & \ge \dfrac{(2-p)(p-\widehat{p})}{2(\widehat{p}p+\widehat{p}-p)} \\
    \dfrac{p}{l} & \ge \Big\{ 1 + \dfrac{(2-p)(p-\widehat{p})}{2(\widehat{p}p+\widehat{p}-p)}\Big\} \\
    \dfrac{1}{l} & \ge \dfrac{1}{p}\Big\{ 1 + \dfrac{(2-p)(p-\widehat{p})}{2(\widehat{p}p+\widehat{p}-p)}\Big\} \ge \dfrac{1}{2}.
\end{align*}
since $l \in(0,2).$ This inequality's positive solution space is $[0,2]$. Combining this with the fact that $p>1$, we have $p \in (1,2].$ 

By the help of all these estimates, we can reduce the PDE system to an ODE given by
\begin{equation}\label{ODE}
    Y_t \le MY-\widetilde{C} Y^{\alpha},
\end{equation}
where $Y=||v||_2$ and as $p\in (1,2]$, we can fix $\alpha =\frac{p}{2}\in (0,1).$ In order to prove that $\exists T^*<\infty$ such that $Y \to 0$ as $T \to T^*$, it is enough to prove that $U = \dfrac{1}{Y} \to \infty$ in finite time $T^*$. On using this substitution we get
\begin{align*}
    -\dfrac{1}{U^2} U_t & = \dfrac{M}{U} - \dfrac{\widetilde{C}}{U^\alpha} \\
    U_t & = \widetilde{C} U^{2-\alpha} - M U.
\end{align*}
One using the fact that $\alpha \in (0,1)$, we have
\begin{equation}\label{ODE_blow}
    U_t = U(\widetilde{C} U^{1-\alpha} - M).
\end{equation}
The above ODE \eqref{ODE_blow} blow up in a finite-time if $U(0)>>(\dfrac{M}{\widetilde{C}})^{\frac{1}{1-\alpha}}$. Hence, we have that if $Y(0)<<(\dfrac{M}{\widetilde{C}})^{\frac{1}{1-\alpha}}$, then \eqref{ODE} goes extinct in a finite time.


\end{proof}

\begin{remark}
One direct implication of the result of Theorem \ref{classic_pde} along with Theorem \ref{FFTE} is that we can say:
The slower diffusing population wins - but a slower diffusing population with a few ($k<<1$) ``very fast" diffusing individuals could loose!
\end{remark}

\begin{remark}
Let's consider the numerical validation of Theorem $\ref{FFTE}$ by considering two cases based on the diffusion coefficient of species $u$ and $v$:
\begin{enumerate}
    \item Let's consider case when $d_1$ is close to $d_2$. Pick $d_1=0.2,d_2=0.199$ and $k=0.999\times 10^{-3}$. For classical case, $p=2$, we have $d_2(1-k)+k<d_1$, so $(u,v) \to (0,v^*)$. But for $p=1.6$, we $(u,v) \to (u^*,0).$ (For other parameter, see Fig.~$\ref{fig:pde_ce_sim1}$)
    \item Let's consider case when $d_1$ and $d_2$ are far-apart. Pick $d_1=1,d_2=10^{-4}$ and $k=0.159$. For classical case, $p=2$, we have $d_2(1-k)+k<d_1$, so $(u,v) \to (0,v^*)$. But for $p=1.6$, we $(u,v) \to (u^*,0).$ (For other parameter, see Fig.~$\ref{fig:pde_ce_sim2}$)
\end{enumerate}
\end{remark}

\begin{figure}[h]
\begin{subfigure}[b]{.475\linewidth}
    \includegraphics[width=\linewidth,height=2.1in]{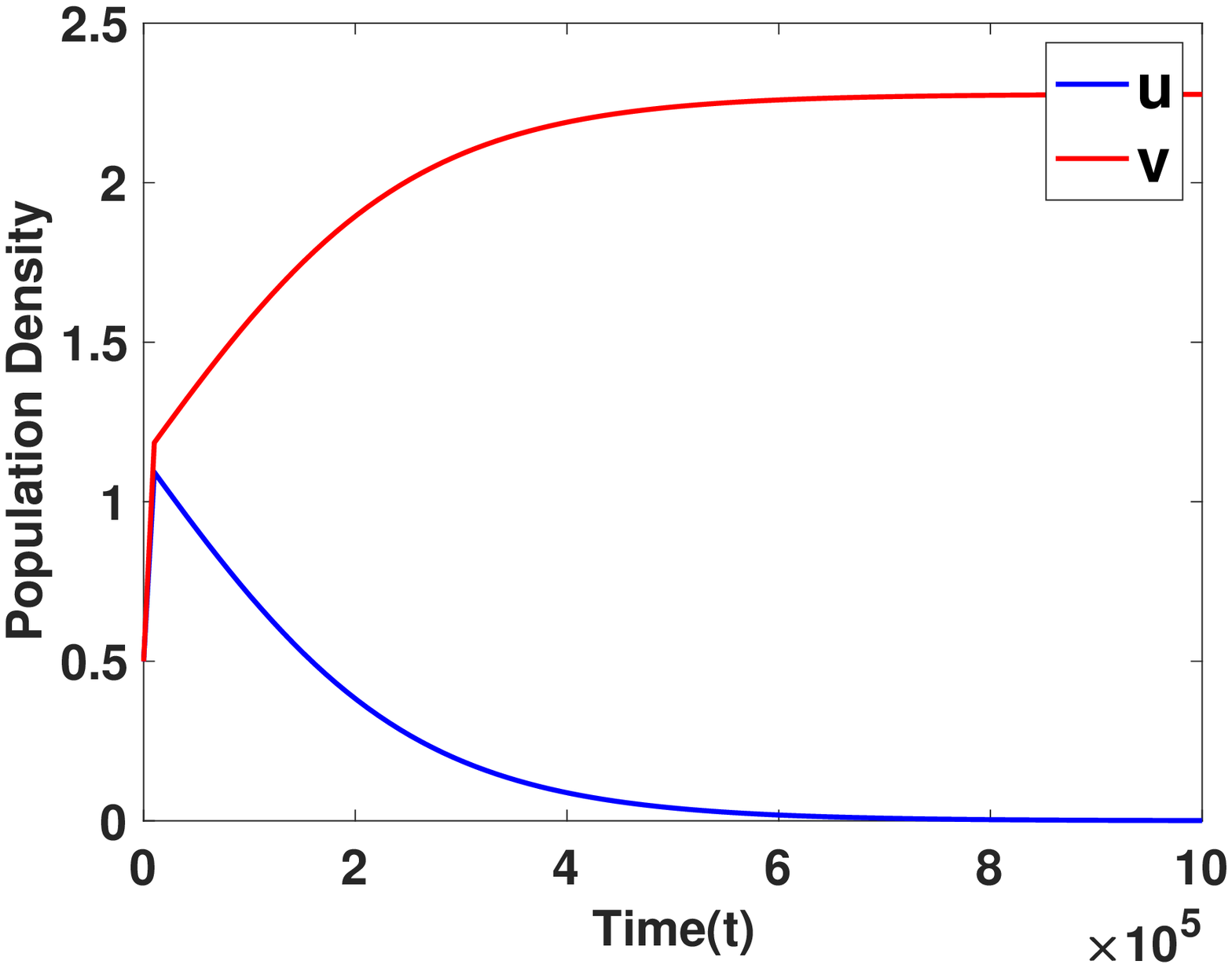}
    \caption{$p=2$}
\end{subfigure}
\hfill
\begin{subfigure}[b]{.475\linewidth}
    \includegraphics[width=\linewidth,height=2.1in]{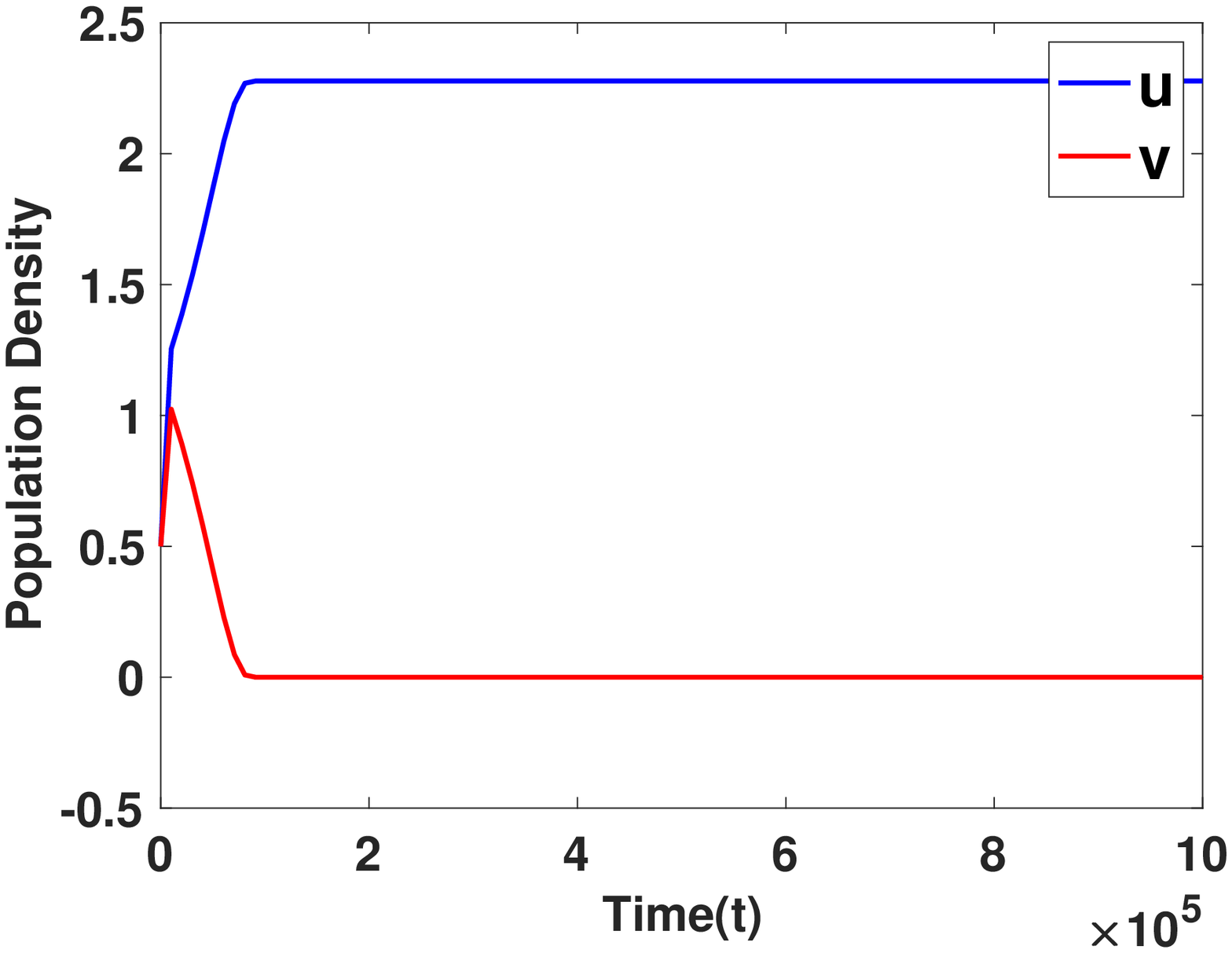}
    \caption{$p=1.6$}
\end{subfigure}
\caption{Numerical simulation of $(\ref{eq:pde_model})$ with $\Omega=[0,1]$ and $m(x)=(3-x)$. The initial data is chosen as $[u_0(x),v_0(x)]=[1.5-x,1.5-x]$, whereas parameters are chosen as $d_1=0.2,d_2=0.199$ and $k=0.99\times 10^{-3}$.}
\label{fig:pde_ce_sim1}
\end{figure}

\begin{figure}[h]
\begin{subfigure}[b]{.475\linewidth}
    \includegraphics[width=\linewidth,height=2.1in]{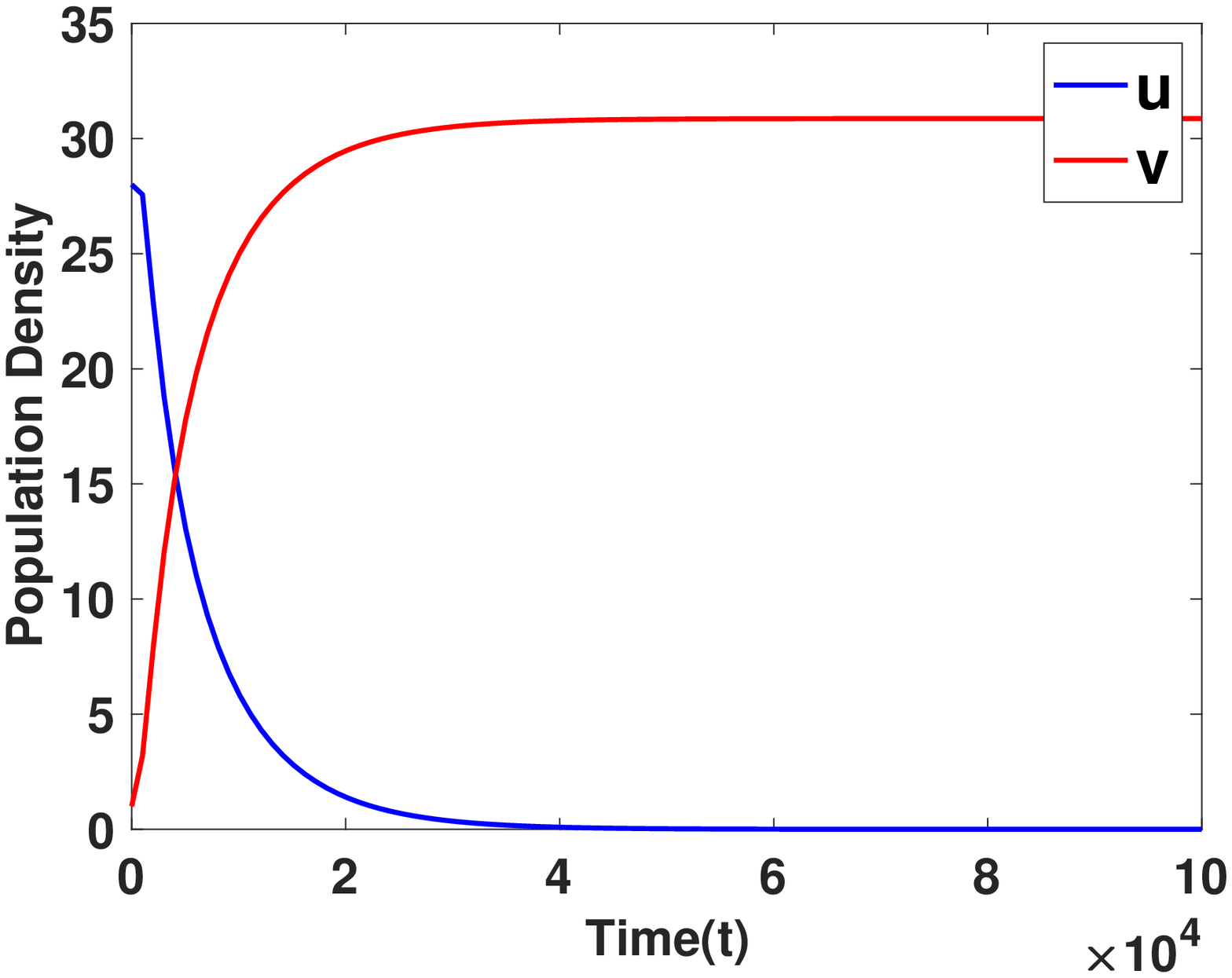}
    \caption{$p=2$}
\end{subfigure}
\hfill
\begin{subfigure}[b]{.475\linewidth}
    \includegraphics[width=\linewidth,height=2.1in]{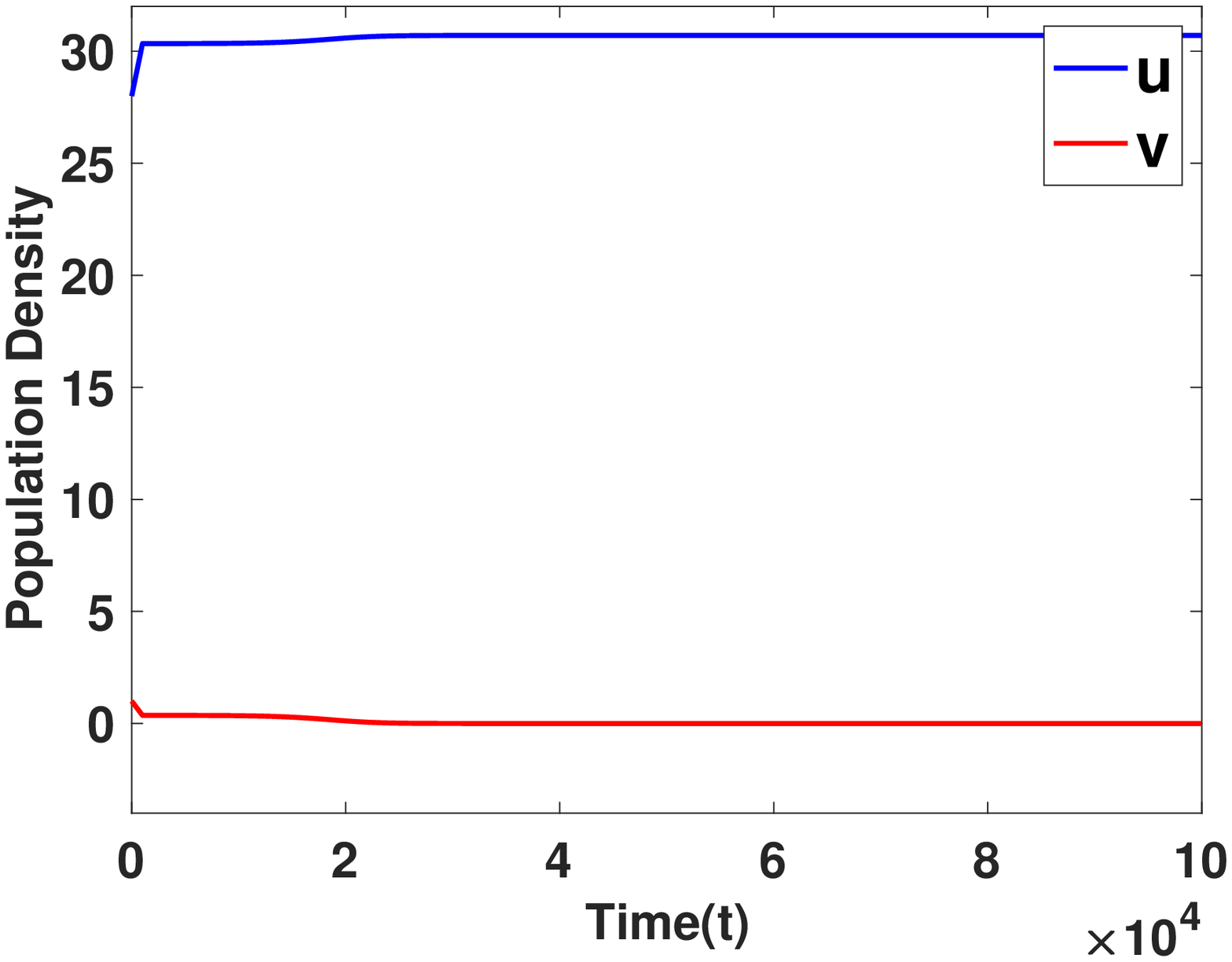}
    \caption{$p=1.6$}
\end{subfigure}
\caption{Numerical simulation of $(\ref{eq:pde_model})$ with $\Omega=[0,1]$ and $m(x)=(30+x^2)$. The The initial data is chosen as$[u_0(x),v_0(x)]=[29-x,x^2]$, whereas parameters are chosen as $d_1=1,d_2=0.0001$ and $k=0.159$.}
\label{fig:pde_ce_sim2}
\end{figure}

\section{Applications}
\subsection{Soybean Aphid Control}

The soybean aphid, \textit{Aphis glycines} (Hemiptera: Aphididae), was first detected in 2000, and since  has become one of the most important insect pests of soybean in the major production areas of the Midwest US(\cite{tilmon2011biology}). It has a heteroecious holocyclic life cycle that utilizes a primary host plant for overwintering and a secondary host plant during the summer. In the spring, aphids emerge and produce three asexual generations on common buckthorn, \textit{Rhamnus cathartica} L., then migrate to soybeans \textit{Glycine max} L.. Aphids continue to reproduce asexually on soybeans, producing as many as 15 generations during the summer (\cite{R11}). In North America, aphids arrive on soybean fields in June, where populations increase by four orders of magnitude and at the end of the growing season (mid-September), aphids begin the migration back to their overwintering host, reproduce sexually and overwinter in the egg stage (\cite{ragsdale2004soybean}). Within Iowa, 40$\%$ of growing seasons from 2004 to 2019, populations of aphids large enough to reduce soybean yield have been observed  with aphid populations peaking in the middle to the end of August (\cite{dean2021developing}).

Colonization and feeding by an insect herbivore can alter the plant’s physiology, favoring the subsequent colonization of additional conspecifics. There are two mechanisms by which this susceptibility can be induced, feeding facilitation and the obviation of resistance. Feeding facilitation is a more general mechanism by which the general physiology of the host plant is altered by the herbivore, often in a density dependent manner (\cite{price2011insect}). \\

A more specific mechanism that inducts susceptibility is the obviation of traits that confer resistance to the herbivore (\cite{baluch2012obviation}; \cite{ent_soc_am_matt}). This mechanism requires a subset of the herbivore population that is virulent, capable of surviving on the resistant genotype of the host plant. By obviating the resistance through a physiological change to the plant, avirulent subpopulations can now survive on the resistant plant. Both mechanisms allow sub-populations that vary by genotype (i.e. virulent and avirulent) to co-exist on resistant host plants.  Both mechanisms have been observed in populations of soybean aphids, colonizing soybean plants of varying genotypes that includes resistance to soybean aphids (\cite{ent_soc_am_matt}).

Field surveys in North America have demonstrated that soybean aphid biotypes can co-occur in the same fields (\cite{cooper2015geographic}, \cite{alt2019geographic}). Laboratory studies have demonstrated that virulent and avirulent biotypes can co-exist on a shared plant for at least 2-3 generations (\cite{ent_soc_am_matt}).

\subsection{Model Formulation}
\subsubsection{Background: The Single Bio Type Case}
 Kindlemann et.al. 2010 (\cite{kindlmann2010modelling}) are among the first to propose a model for the population dynamics of Aphids  using a set of differential equations,
\begin{eqnarray}
\label{eq:m1}
  \frac{dh}{dt}&=&ax ;  h(0)=0 \nonumber\\ 
     \frac{dx}{dt}&=&(r-h)x \ ;  x(0)=x_0\nonumber\\
     \nonumber
\end{eqnarray}

where $h(t)$ is the cumulative population density of a single aphid biotype at time $t$; $x(t)$ is the population density at time $t$. $a$ is a scalar constant and $r$ is the growth rate of the aphids. The aphid population initially rises due to the linear  growth term, but as the cumulative density becomes greater than the growth rate $r$, the population is brought down, due to the effects of competition. This results in a hump-shaped population density over time, typical of a boom-bust type scenario (\cite{kindlmann2010modelling}). This is an apt description of aphid dynamics, particularly when one is interested in exploring soybean aphid dynamics on soybeans during the summer growing season. The type of population growth described by this model has been observed in soybean aphids in North America, with colonisation in June, then a gradual build-up of the population with a high peak in August and then a reduction, with all of the Aphids dispersing by September, for overwintering.\\ 

The model in \eqref{eq:m1}, is quite different from the classical logistic growth model, which predicts growth to a certain carrying capacity. It is an example of a non-autonomous model, wherein the right-hand side of the differential equation depends explicitly on time. The rigorous mathematical analysis of such systems is quite involved, and the methods of classical autonomous systems do not apply. Hence the rigorous dynamical analysis of \eqref{eq:m1} is not found in the literature. However, it provides a starting point to model more intricate aphid dynamics, particularly when a species presents two or more biotypes.

\subsubsection{Virulent and avirulent aphids: Two Bio types}

Sub-populations of an herbivorous species can be organized into biotypes, defined as genotypes capable of surviving and reproducing on host plants containing traits (e.g., antibiosis and or antixenosis) conferring resistance to that herbivore (\cite{downie2010baubles}). Specifically, for the soybean aphid, biotypes are classified based on their ability to colonize soybean varieties expressing \textit{Rag}-genes(Rag is derived from the expression, Resistance to \textit{Aphis glycines}). For example, soybean aphid biotype 1 is susceptible to all \textit{Rag}-genes, therefore it is called avirulent. Biotype 2 is virulent to \textit{Rag}1 (\cite{kim2008discovery}), biotype 3 is virulent to \textit{Rag}2 (\cite{hill2010new}), and biotype 4 is virulent to both \textit{Rag}1 and \textit{Rag}2, capable of surviving on plants with these genes either alone and together (\cite{alt2013soybean}). These four soybean aphid biotypes have been found throughout the soybean production areas in the Midwest US (\cite{cooper2015geographic}, \cite{alt2019geographic}).\\

If one considers a resistant soybean plant, where virulent and avirulent aphids are both trying to colonize, various dynamics are at play, that \emph{cannot} be described by earlier models. First, the virulent and avirulent are in direct competition for space, similar to interspecies competition. The virulent aphids are also in competition for space with other virulent aphids, as avirulent aphids are in competition for space with other avirulent aphids, similar to the intraspecies competition. These are direct effects of competition. Note, on a resistant plant both the avirulent and virulent aphids are able to weaken the plant's defenses via feeding facilitation. However, for the avirulent aphid this only occurs if it arrives in sufficiently large numbers (\cite{ent_soc_am_matt}). Thus there is a definite resistant level in the plant that is dependent on initial avirulent aphid density. If the avirulent aphids arrive in sufficient numbers above this level, they could colonise a resistant plant, but below this will die out - this is very similar to a resistance effect in ecology. Note the virulent biotype alters the plant by obviating the resistance, removing its impact on both virulent and avirulent aphids. This removal of the plants' resistance level by the virulent biotype, eases the colonisation process for the avirulent biotype - this then is an indirect form of cooperation at play. Thus, the plants' resistance is a dynamic process, dependent on the presence and densities of these biotypes. The following model is proposed to capture these dynamics,

\begin{eqnarray}
\label{aphid_Classic}
  \frac{dh}{dt}&=&a(x_A+ x_V ) \nonumber\\ 
    \frac{dx_A}{dt}&=&(r-h)(x_A-A) \nonumber\\ 
  \frac{dx_V}{dt}&=&(r-h)x_V  \nonumber\\
   \frac{dA}{dt}&=&-(k_r x_V+k_f x_V+k_f sgn(x_A-R) x_A)A \nonumber  \nonumber 
\end{eqnarray}
    Here \textit{$x_A(t)$} refer to avirulent aphid population density and \textit{$x_V(t)$} refers to the virulent aphid population density \cite{Evol:matt}; 
    \textit{$h$} is the combined cumulative population density of both avirulent and virulent aphids respectively at time \textit{t};
    \textit{r} is the maximum potential growth rate of the Aphid;\\ 
    \textit{a} is a scaling constant relating prey cumulative density to its own dynamics;
    \textit{A} is the dynamic resistance threshold of the plant \cite{thesis:berec}. This decreases due to both avirulent and virulent aphid density, that is \textit{$x_V$} and \textit{$x_A$}. It is measured in the same units as Aphid density.\\
    \textit{$k_f$} is the rate of feeding facilitation constant and \textit{$k_r$} is the rate of obviation of resistance\cite{ent_soc_am_matt}
    \textit{R} is the threshold population density above which an avirulent population can have a feeding facilitation effect on the plant and below it, the effect of feeding facilitation because of avirulent aphid is zero;\\
    \textit{sgn(x)} is a function returning 0 or 1. It returns 1 if $x>0$ or else it returns 0; This function regulates whether avirulent aphids have enough initial density to have the effect of feeding facilitation on the plant.
    Previous studies have demonstrated that obviation of resistance is much more effective in shutting down the resistance than feeding facilitation \cite{ent_soc_am_matt}. Therefore, \textit{$k_r$}$>$\textit{$k_f$} whenever both effects take place simultaneously.

    \subsubsection{Harvesting of avirulent aphids}
        If a similar harvesting technique is used as in system \ref{Eqn:2} and a particular aphid population is harvested then the dynamics for aphid bio-type dominance in the system can change. We modify the system \ref{aphid_Classic} with fast diffusion on each aphid bio-type separately and study the system.
        
     The Virulent aphid are harvested out of the system \ref{aphid_Classic} which results in the modified model 
    \begin{eqnarray}
    \label{aphid_vir_eq}
    \frac{dh}{dt}&=&a(x_A+ x_V ) \nonumber\\ 
    \frac{dx_A}{dt}&=&(rq-h)(x_A-A)- (1-q)x_A^p \nonumber\\ 
    \frac{dx_V}{dt}&=&(r-h)x_V  \nonumber\\
    \frac{dA}{dt}&=&-(k_r x_V+k_f x_V+qk_f sgn(x_A-R) x_A)A 
    \end{eqnarray}

\begin{figure}[H]
    \begin{subfigure}[b]{.475\linewidth}
        \includegraphics[width=\linewidth]{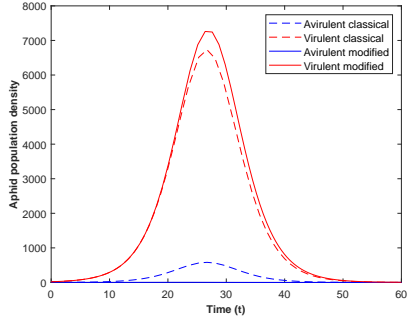}
        \caption{The initial conditions used for the time series are $(h(0),x_A(0),x_V(0),A(0)) = (0,15,20,30)$ }
        \label{low_Avi_low_vi}
    \end{subfigure}
    \hfill
    \begin{subfigure}[b]{.475\linewidth}
        \includegraphics[width=\linewidth]{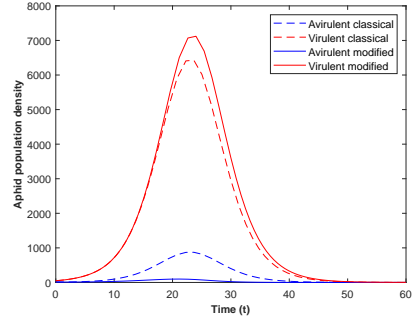}
        \caption{The initial conditions used for the time series are $(h(0),x_A(0),x_V(0),A(0)) = (0,15,50,30)$}
        \label{low_avi_high_vi}
    \end{subfigure}

    \begin{subfigure}[b]{.475\linewidth}
        \includegraphics[width=\linewidth]{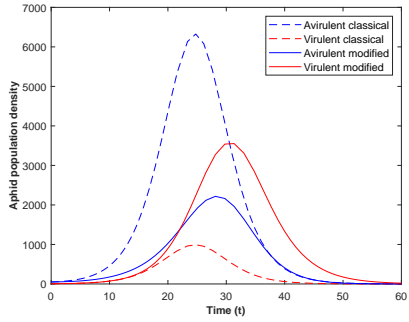}
        \caption{The initial conditions used for the time series are $(h(0),x_A(0),x_V(0),A(0)) = (0,50,5,30)$. }
        \label{high_Avi_low_vi}
    \end{subfigure}
    \hfill
    \begin{subfigure}[b]{.475\linewidth}
        \includegraphics[width=\linewidth]{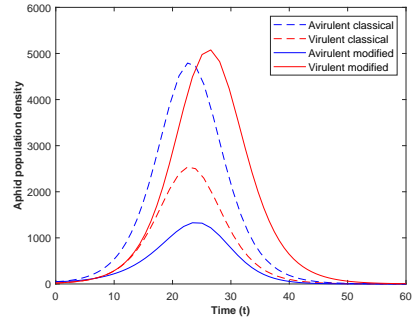}
        \caption{The initial conditions used for the time series are $(h(0),x_A(0),x_V(0),A(0)) = (0,50,20,30)$ }
        \label{high_both_1}
    \end{subfigure}
    \hfill
    \begin{subfigure}[b]{.475\linewidth}
        \includegraphics[width=\linewidth]{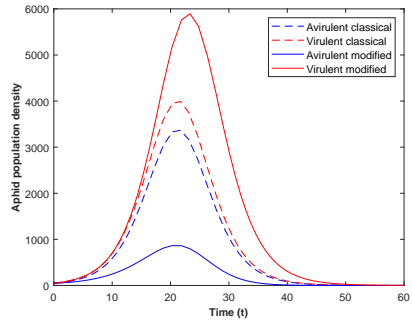}
        \caption{The initial conditions used for the time series are $(h(0),x_A(0),x_V(0),A(0)) = (0,50,50,30)$ }
        \label{high_both}
    \end{subfigure}

    \caption{The dynamics of the virulent aphid($x_V$) and avirulent aphid $(x_A)$ is shown in fig \ref{low_Avi_low_vi} - \ref{high_both}.The dotted lines represent the dynamics of system \ref{aphid_Classic} and the dense line represents the dynamics of system \ref{aphid_vir_eq}. The red colour represents virulent aphid ($x_V$) and blue colour represents the avirulent aphid population ($x_A$). Figures \ref{low_Avi_low_vi}-\ref{high_both} have the same parameter set $r=0.27, a=0.000005, k_f=0.001, k_r=0.01, R=30, k=1.$ }
    \label{aphid_vir}
\end{figure}

\subsection{Effect of avirulent aphid harvesting}
\hfill\\
    In the figures \ref{aphid_vir} the effect of harvesting avirulent aphids is seen compared to the original results
    \begin{enumerate}
        \item Harvesting of avirulent aphids decreases the peak population of avirulent aphid as seen in the figures. When the avirulent aphid population is below the initial threshold the avirulent aphids go to extinction
        \item The peak population for both the biotypes of aphid occurs at different times as can be seen in figure \ref{high_Avi_low_vi}. 
    \end{enumerate}

\section{Appendix}

\begin{proof} of lemma \ref{Boundary_v}: 
When $\Bar{u}=0$ when studying the boundary equilibrium $E_v(0,\Bar{v})$, the nullcline (\ref{nullcline}) simplifies to $\phi(\Bar{v})$ given by,\\
\\
$\phi(v)=a_q-b_2\Bar{v}-(1-q)\Bar{v}^{p-1}$

To study the roots of the polynomial $\phi(\Bar{v})$, we study the monotonicity of the polynomial which is described by,\\
$\phi^{'}(v)= -b_2 + (1-q)(1-p)v^{p-2}$.
\\
Thus, the extrema is attainable when $\phi^{'}(v)=0$. We see, it is possible to attain extrema at $v=v_{\phi}$ where $v_{\phi}^{p-2}= \frac{b_2}{(1-q)(1-p)} >0 $ as $b_2>0$ and $0 < p,q < 1$.

As $\phi^{''}(v_{\phi}) = (1-q)(1-p)(p-2)v_\phi^{p-3}>0$, we have a maxima at $v=v_\phi$. Now as $v_\phi$ is unique for the fixed parameter set so in the invariant set we have unique maxima at $v=v_\phi$. \\
As there exist one maxima, we can deduce there can be at most two roots of $\phi(\Bar{v})$ and may have no roots depending on the functional value $\phi(v_{\phi})$.
Computing we get,
$\phi(v_{\phi})= a_2q-b_2v_\phi(\frac{2-p}{1-p})$.\\
If $\phi(v_{\phi})>0$ i.e. $v_\phi < \frac{a_2q(1-p))}{b_2(2-p)}$ then there exists two boundary equilbria. When $v_\phi > \frac{a_2q(1-p))}{b_2(2-p)}$ then there are no boundary equilibrium of the form $E_v(0,\Bar{v})$.

\end{proof}

\begin{proof} of lemma \ref{b_1b_2-c_1c_2<0 exists}: 
 Let there exist a positive $v^*$ which is a root of the polynomial $\phi(v)$ given by,\\ 
 
$\phi(v^*)=(a_2b_1q-c_2a_1)+(c_1c_2-b_1b_2)v^*-b_1(1-q)v^{*(p-1)}$ where $b_1b_2-c_1c_2 \leq 0$.\\

To understand the monotonicity of the curve we study the slope given by,\\
 $\phi'(v^*)=(c_1c_2-b_1b_2)-b_1(p-1)(1-q)v^{*(p-2)}$
We know, $0<p,q<1$ and $v^*>0$ so $b_1(p-1)(1-q)v^{*(p-2)}>0$. Thus, $\phi'(v^*)>0$ which shows $\phi(v^*)$ is a monotonically increasing curve.\\
For $v\to 0$, $\phi(v) \to  -\infty$ and for $v\to \infty$, $\phi(v) \to  \infty$ so $\min(\phi(v))<0$.\\
 
Thus, there exists a positive $c_1$ such that $\phi(c_1)<0$ and $\exists$ positive $c_2$ such that $\phi(c_2)>0$. By using the Mean Value theorem  we can conclude that there exists a  $v^* \in (c_1,c_2)$ such that $\phi(v^*)=0$ and as $\phi(v)$ is monotonically increasing so $v^*$ is unique.
 
 Thus if $(b_1b_2-c_1c_2) \leq 0$ then there can only be one positive $v^*$ such that $\phi(v^*)=0$.
 
 \end{proof}

\begin{proof} of lemma \ref{stab_interior}: 
    According to lemma \ref{Boundary_v} we can have two, one or no boundary equilibria of the form $E_v(0,\Bar{v})$ depending on the functional value of the nullcline. 
We study when we have two boundary equilibria $E_{v_1}(0,\Bar{v_1})$ and $E_{v_2}(0,\Bar{v_2})$. Without loss of generality, lets assume $v_1<v_2$. By lemma \ref{Boundary_v} we know $v_\phi$ is the maxima of the nullcline. By Mean Value Theorem, we can assume  $v_1< v_\phi < v_2$.\\
To study the stability of the equilibria we simplify the Wronskian matrix and have\\
$W(E_v)=\begin{bmatrix}
a_1-c_1v & -c_1u \\
-c_2v & -b_2v+(1-p)(1-q)v^{p-1}
\end{bmatrix}$.

The eigenvalues for the system are $\lambda_1=a_1-c_1v$ and $\lambda_2=v(-b_2+(1-p)(1-q)v^{p-2})$.\\
We study the case when there exist two boundary equilibria.
When $v_2>v_\phi$ where $v_{\phi}^{p-2}=\frac{b_2}{(1-q)(1-p)}$, then \\
$\lambda_2=v_2(-b_2+(1-p)(1-q)v_2^{p-2})$\\
$< v_2(-v_\phi^{p-2}(1-p)(1-q)+(1-p)(1-q)v_2^{p-2})$.\\
$=v_2(1-p)(1-q)(v_2^{p-2}-v_\phi^{p-2})$.\\
As $v_\phi<v_2$ and $p-2<0$ we have $\lambda_2 < 0$.\\

If we simplify $\lambda_1= a_1-c_1v_1$, then $\lambda_1<0$ if $v_2 > \frac{a_1}{c_1}$. Thus $E_{v_2}(0,v_2)$ is stable node when $v_2>a_1/c_1$.\\

The determinant of wronskian is given by at $E_{v_1}$,\\
$det(W(E_{v_1}))=v_1(a_1-c_1v_1)(-b_2+(1-p)(1-q)v_1^{p-2})$.\\
As $v_1<v_\phi$ so, $det(W(E_{v_1})) = v_1(a_1-c_1v_1)(-b_2+(1-p)(1-q)v_2^{p-2})$\\
$< v_1(a_1-c_1v_1)(1-p)(1-q)(v_\phi^{p-2}-v_1^{p-2})=0$.\\
Thus $E_{v_1}(0,v_1)$ is a saddle point. 


\end{proof}

\begin{proof} of lemma \ref{unstab_interior}: 
According to lemma \ref{b_1b_2-c_1c_2<0 exists} if $(b_1b_2-c_1c_2)\leq 0$ there exists an unique equilibrium $E^*(u^*,v^*)$. The simplified Jacobian matrix for the unique equilibrium is given by,\\
$\begin{bmatrix}
b_1u & -c_1u \\
-c_2v & -b_2v+(1-p)(1-q)v^{p-1}
\end{bmatrix}$

The determinant of the Jacobian matrix can be given by,\\
$(-b_1u^*)(-b_2v^*+(1-p)(1-q)v^{*(p-1)}-c_1c_2u^*v^*$\\
By simplifying we get,
$det(J(E^*))=(b_1b_2-c_1c_2)v^*-b_1(1-p)(1-q)v^{*(p-1)}$. As $(b_1b_2-c_1c_2)\leq 0$ and $0<p,q<1$ so det$(J(E^*))$ is always negative.\\
Thus $E(u^*,v^*)$ is always an unstable point.
\end{proof}

\begin{proof} of lemma \ref{stable_equi}: 
According to lemma \ref{th:existence} if $(c_1c_2-b_1b_2)< 0$ and $(a_2b_1q-c_2a_1)+(c_1c_2-b_1b_2)(1+\frac{1}{1-p})v_{max}>0$ where $v_{max}=-\frac{(1-q)(1-p)}{(c_1c_2-b_1b_2)}^{\frac{1}{2-p}}$ then there exists two interior equilibrium points.
We first study when we can get a stable interior equilibrium.\\
The conditions to be met for a stable equilibrium are:\\
\begin{itemize}
    \item $(a_2b_1q-c_2a_1)+(c_1c_2-b_1b_2)(1+\frac{1}{1-p})v_{max}>0$ (Existence condition)
    \item Trace= $-b_1u^*-b_2v^*+(1-p)(1-q)v^{*(p-1)}<0$
    \item Determinant= $(-b_1u^*)(-b_2v^*+(1-p)(1-q)v^{*(p-1)}-c_1c_2u^*v^*>0$
\end{itemize}
By simplifying the trace we get 

$(1-q)V^{*(p-1)}<\frac{a_1+(b_2-c_1)v^*}{1-p}$

$qa_2b_1-a_1c_2+(c_1c_2-b_1b_2)v^*<\frac{a_1+(b_2-c_1)v^*}{1-p}$

So, $q<\frac{1}{a_2b_1(1-p)}(a_1+a_1c_2(1-p)+((b_2-c_1)+(b_1b_2-c_1c_2)(1-p))v^*)$

Simplifying determinant we get,\\
$(1-q)v^{*(p-1)}<\frac{(b_1b_2-c_1c_2)v^*}{(1-p)b_1}$

$qa_2b_1-a_1c_2+(c_1c_2-b_1b_2)v^*<\frac{(b_1b_2-c_1c_2)v^*}{(1-p)b_1}$

$q<\frac{1}{a_2b_1(1-p)}(a_1c_2(1-p)+(b_1b_2-c_1c_2)\frac{2-p}{b_1}v^*)$

Combining the results we can say that the interior equilibrium point is locally stable if \\
$q<\frac{1}{a_2b_1(1-p)}\min(a_1+a_1c_2(1-p)+((b_2-c_1)+(b_1b_2-c_1c_2)(1-p))v^*,a_1c_2(1-p)+(b_1b_2-c_1c_2)\frac{2-p}{b_1}v^*)$.

From the existing condition we get that,
$q>\frac{1}{a_2b_1(1-p)}((b_1b_2-c_1c_2)(2-p)v_{max}+c_2a_1(1-p))$

\end{proof}

\end{document}